\documentclass[11pt]{article}

\usepackage{latexsym}
\usepackage{amsthm}
\usepackage{amsmath}
\usepackage{amsfonts, epsfig, amsmath, amssymb, color, amscd}
\usepackage[cmtip,arrow]{xy}
\usepackage{pb-diagram, pb-xy}
\usepackage{amssymb,epsfig}
\usepackage{color}
\usepackage[cmtip,arrow]{xy}
\usepackage{pb-diagram, pb-xy}
\usepackage{amssymb,epsfig,amsfonts}

\newfont{\sdbl}{msbm9}
\newfont{\dbl}{msbm10 at 12pt}
\theoremstyle{definition}
\newcommand{\da}{{\mbox{\dbl A}}}
\newcommand{\cn}{{{\cal N}}}
\newcommand{\ci}{{{\cal I}}}
\newcommand{\cb}{{{\cal B}}}
\newcommand{\ca}{{{\cal A}}}
\newcommand{\cl}{{{\cal L}}}
\newcommand{\cg}{{{\cal G}}}
\newcommand{\dpp}{{\mbox{\dbl P}}}
\newcommand{\dz}{{\mbox{\dbl Z}}}

\newcommand{\dn}{{\mbox{\dbl N}}}

\newcommand{\sdz}{{\mbox{\sdbl Z}}}
\newcommand{\sdn}{{\mbox{\sdbl N}}}
\newcommand{\dc}{{\mbox{\dbl C}}}
\newcommand{\dq}{{\mbox{\dbl Q}}}

\newcommand{\ord}{\mathop{\rm ord}\nolimits}

\newcommand{\gr}{\mathop {\rm gr}}

\newcommand{\Der}{\mathop {\rm Der}}
\newcommand{\Pic}{\mathop {\rm Pic}}
\newcommand{\End}{\mathop {\rm End}}

\newcommand{\Div}{\mathop {\rm Div}}
\newcommand{\WDiv}{\mathop {\rm WDiv}}
\newcommand{\Ker}{\mathop {\rm Ker}}

\newcommand{\Ch}{\mathop {\rm Ch}}

\newcommand{\Quot}{\mathop {\rm Quot}}

\newcommand{\dpth}{\mathop {\rm depth}}
\newcommand{\rk}{\mathop {\rm rk}}
\newcommand{\heit}{\mathop {\rm ht}}
\newcommand{\trdeg}{\mathop {\rm trdeg}}

\newcommand{\Spec}{\mathop {\rm Spec}}
\newcommand{\Spf}{\mathop {\rm Spf}}
\newcommand{\Specz}{\mathop {\bf  Spec}}
\newcommand{\Proj}{\mathop {\rm Proj}}

\newcommand\limproj{\mathop{\underleftarrow{\lim}}}
\newcommand\limind{\mathop{\underrightarrow{\lim}}}
\newcommand{\xo}{{\mbox{\em \r{X}}}}

\newcommand{\Hom}{\mathop{\rm Hom}}

\makeatother
\newtheorem{defin}{Definition}[section]
\newtheorem{nt}{Remark}[section]
\newtheorem{pb}{Problem}[section]
\newtheorem{ex}{Example}[section]

\theoremstyle{plain}
\newtheorem{prop}{Proposition}[section]
\newtheorem{theo}{Theorem}[section]
\newtheorem{lemma}{Lemma}[section]
\newtheorem{corol}{Corollary}[section]

\newtheorem{lemmaB}{Lemma}
\newtheorem{ntB}{Remark}

\newcommand{\eqdef}{\stackrel{\rm def}{=}}

\newcommand{\lto}{\longrightarrow}
\newcommand{\Ord}{\mathop {\rm \bf ord}}
\newcommand{\co}{{{\cal O}}}
\newcommand{\cnn}{{{\cal N}}}
\newcommand{\cf}{{{\cal F}}}
\newcommand{\cq}{{{\cal Q}}}

\newcommand{\cm}{{{\cal M}}}

\title{Commuting  differential operators and higher-dimensional algebraic varieties}
\author{Herbert Kurke, \quad Denis Osipov, \quad Alexander Zheglov}

\date{}

\begin{document}

\maketitle

\quad \qquad \qquad
{\em Dedicated to A.~N.~Parshin on the occasion of his 70th birthday}

\begin{abstract}
Several algebro-geometric properties of commutative rings of partial differential operators as well as several geometric constructions are investigated. In particular, we show how to associate a geometric data by a commutative ring of partial differential operators, and we investigate the properties of these geometric data. This construction is in some sense similar to the construction of a formal module of Baker-Akhieser functions. On the other hand, there is a recent generalization of Sato's theory which belongs to the third author of this paper. We compare both approaches to the commutative rings of partial differential operators in two variables. As a by-product we get several necessary conditions on geometric data describing commutative rings of partial differential operators.

\vspace{0.5cm}


\vspace{0.5cm}


\end{abstract}

\section{Introduction}

In this paper we study some algebro-geometric properties of commutative rings of partial differential operators (PDO for short).

One of very complicated and intriguing questions appearing in the theory of algebraically integrable systems is how to find explicit examples of certain commutative rings of PDOs. In the case of $n=1$ variable (or, in other words, in the case of commuting ordinary differential operators) this question is related to the method of constructing explicit solutions for various nonlinear integrable equations, for instance KdV or KP. In this case the question is: how to find a ring of commuting ordinary differential operators that contains a pair of monic operators $P,Q$ such that $\dc [P,Q]\not\simeq \dc [u]$? The classification of such rings in the case when the orders of $P,Q$ are relatively prime was obtained already in works of Burchnall and Chaundy \cite{BC} (see also \cite{Ba}) by purely algebraic methods, though the reconstruction of the coefficients of operators was not sufficiently effective. The classification of such rings in general case was given in the work of Krichever \cite{Kr1}, where the connection of this classification with the theory of integrable systems, with the spectral operator theory and with the theory of linear differential equations with periodical coefficients was pointed out. At the same time there appeared a rich theory connected with famous equations such as KP, KdV, sin-Gordon, Toda, etc. (the so called KP theory or the finite-gap theory; for reviews see for example \cite{Kr}, \cite{Man}, \cite{D}, \cite{SW}, \cite{Mul}, \cite{DKN} and references therein).

In the case of operators in $n>1$ variables the question mentioned above can be formulated as follows (see \cite{Kr},\cite{ChV},\cite{ChV3}, cf. also \cite{Ch},\cite{BEGa}):  how to find a ring of commuting partial differential operators that contains $n+1$ operators $L_0,\ldots ,L_n$ with algebraically independent homogeneous constant highest symbols $\sigma_1,\ldots ,\sigma_n$ such that $\dc [\dc^n]$ is finitely generated as a module over the ring generated by $\sigma_1,\ldots ,\sigma_n$ and  $L_0$ is not a polynomial combination of $L_1,\ldots ,L_n$? In the work \cite{Kr} certain statements from $n=1$ case were generalized for such rings.
 In particular, an analogue of the Burchnall-Chaundy lemma was proved there. This analogue says
 that $n+1$ commuting operators $L_0,\ldots ,L_n$ are algebraically dependent. Besides,
it was shown that for a ring of commuting partial differential operators satisfying certain "generic  position" assumptions  there is a unique Baker-Akhieser function (a generator of the module consisting of eigenfunctions of the ring of PDO) that completely characterizes the ring by its spectral variety (in other words, the module is free of rank one). There was also offered a compactification of the spectral surface.

Unfortunately, up to now there are known only a few examples of such rings. The first nontrivial examples  appeared in \cite{ChV}, \cite{ChV2}, \cite{ChV3}. The examples were connected  with the quantum (deformed) Calogero-Moser systems.
Later the ideas of these constructions  were developed in a series of papers (see e.g. \cite{FV}, \cite{FV2}, \cite{EG}) in order to construct more examples (for review see e.g. \cite{Ch} and references therein; cf. also \cite{BEG}, \cite{BEGa}, \cite{BeK}). Let's also mention that the idea to construct a free BA-module (the module consisting of eigenfunctions of the ring of PDO) was developed later by various authors (see e.g. \cite{Na}, \cite{Mi}, \cite{Ch}) to produce explicit examples of commuting matrix rings of PDO.

The results about the commutative rings of PDO from \cite{Kr} can be reformulated by saying that there is a construction that associates to such a ring of commuting operators some algebro-geometric data that consists of a completed (projective) affine spectral variety, the divisor at infinity, a  torsion free sheaf of rank one (or even a family of such sheaves) and some extra trivialisation data (cf. the work \cite{ZhM}, where this construction is written in rank one case in terms of a family of Krichever sheaves).
This geometric construction is an analogue of the construction coming from $n=1$ case.

Unfortunately, unlike the $n=1$ case, it is not known yet which geometric data describe exactly the commutative rings of PDOs. But these data describe commutative rings of completed PDOs in the $n=2$ case.
Namely, for differential operators in two variables  the following approach was offered in \cite{Zhe2}. We consider a wider class of operators: the operators from the completed ring $\hat{D}$ of differential operators (see \cite[Sec.2.1.5]{Zhe2}; see also section \ref{rings} where we recall several definitions and results from loc.cit.). We would like to emphasize that it is not the ring of pseudo-differential operators! The operators from this ring contain all usual partial differential operators, and difference operators as well. They are also linear and act on the ring of germs of analytical functions. In the work \cite{Zhe2} all commutative subrings in $\hat{D}$  satisfying certain mild conditions are classified in terms of Parshin's modified geometric data (such data   include algebraic projective surface, an ample $\dq$-Cartier divisor, a point regular on this divisor and on the surface, a torsion free sheaf on the surface and some extra trivialisation data). Let's mention that such rings contain all subrings of partial differential operators in two variables considered above after an appropriate change of variables (see section 3.1 of \cite{Zhe2}). The approach offered in \cite{Zhe2} generalizes the approach of Sato in dimension one  and differs from the approach connected with the study of the Baker-Akhieser functions. This approach is closely related with the generalization of the Krichever map offered by Parshin in \cite{Pa1}, \cite{Pa}. Let's explain  the history connected with this approach in more details.

As we have already mentioned above, the commutative rings of ODOs are classified in terms of geometric data, whose main geometric object is a projective curve. In the classical KP theory there is a map which associate to each such data a pair of subspaces $(A,W)$ ("Schur pairs") in the space $V=k((z))$, where $A\varsupsetneq k$ is a stabilizer $k$-subalgebra of $W$ in $V$: $A\cdot W\subset W$, and $W$ is a point of the infinite-dimensional Sato grassmannian (see e.g. \cite{Mul} for details). This map is usually called as the Krichever  map in the literature. In works \cite{Pa1}, \cite{Pa} (see also \cite{Os}) Parshin introduced an analogue of the Krichever map which associates to each geometric data (which include a Cohen-Macaulay surface, an ample Cartier divisor, a smooth point
and a vector bundle) a pair of subspaces $(\mathbb{A},\mathbb{W})$ in the two-dimensional
local field associated with the flag (surface, divisor, point) $k((u))((t))$ (with analogous properties). He showed that this map is injective on such data. In works \cite{Pa}, \cite{Os} some combinatorial
construction was also given. This construction helps to calculate cohomology groups of vector bundles in
terms of these subspaces and permits to reconstruct the geometric data from the pair $(\mathbb{A},\mathbb{W})$. The difference of this new Krichever-Parshin map from the Krichever map is that the last map is known to be bijective.

To extend the Krichever-Parshin map and to make it bijective we introduced in the work \cite{Ku} (see also section \ref{ribbons}, where we recall some definitions) new geometric objects called formal punctured ribbons (or simply ribbons for short) and torsion free coherent sheaves on them, we extended this map on the set of new geometric data which include these objects and showed the bijection between the set of geometric data and the set of pairs of subspaces $(\mathbb{A},\mathbb{W})$ (also called  generalized Schur pairs) satisfying certain combinatorial conditions. We also showed that for any given Parshin's geometric data one can construct a unique geometric data with a ribbon, and the initial Parshin's data can be reconstructed from the new data with help of the combinatorial construction mentioned above.

At the same time in the work \cite{Pa0} Parshin offered to consider a multi-variable analogue of the KP-hierarchy which, being modified, is related to algebraic surfaces and torsion free sheaves on such surfaces as well as to a wider class of geometric data consisting of ribbons and torsion free sheaves on them if the number of variables is equal to two (see \cite{Zhe},  \cite[Introduction]{Ku1}). Thus, in the work \cite{Ku1} we described the geometric structure of the Picard scheme of a ribbon. This scheme has a nice group structure and can be thought of as an analogue of the Jacobian of a curve in the context of the classical KP theory. In particular, generalized  KP flows are defined on such schemes.

To classify commutative subrings in the ring $\hat{D}$ in the already mentioned work \cite{Zhe2} the Parshin geometric data were modified: now the surface need not be Cohen-Macaulay, the ample divisor need not be Cartier and the sheaf need not be a vector bundle. Moreover, the classification was established there not only in terms of the modified Parshin data, but also in terms of modified Schur pairs. These are the pairs of subspaces $(A,W)$ in the space $k[[u]]((t))$ satisfying  properties similar to the properties of the Schur pairs $(\mathbb{A},\mathbb{W})$ (see sections \ref{catgemdat} and \ref{mapxi} for more details).

At this point the following natural questions appear:\\
1) Can we extend the construction that associates to each Parshin's data the data with ribbon to the set of modified Parshin's data?\\
2) If yes, what happens if we apply the combinatorial construction that reconstructs the Parshin data from its ribbon's data to  ribbon's data coming from modified Parshin's data? \\
3) What is the relationship between the Schur pairs $(\mathbb{A},\mathbb{W})$ and modified Schur pairs $(A,W)$?

The aim of this paper, except the main aim (formulated in the beginning) to study algebro-geometric properties of commutative rings of PDOs, is to give also the answers on these questions. These answers will be important in order to apply the theory of ribbons developed in \cite{Ku}, \cite{Ku1} to study rings of PDO and their isospectral deformations.

The following picture illustrates the relationship between all above mentioned data (we give here a slightly imprecise account to make the exposition easier):
\begin{equation}
\label{intersection}
\mbox{If $A$ is Cohen-Macaulay ring then $A=\mathbb{A}\cap k[[u]]((t))$, $W=\mathbb{W}\cap k[[u]]((t))$}
\end{equation}
\clearpage
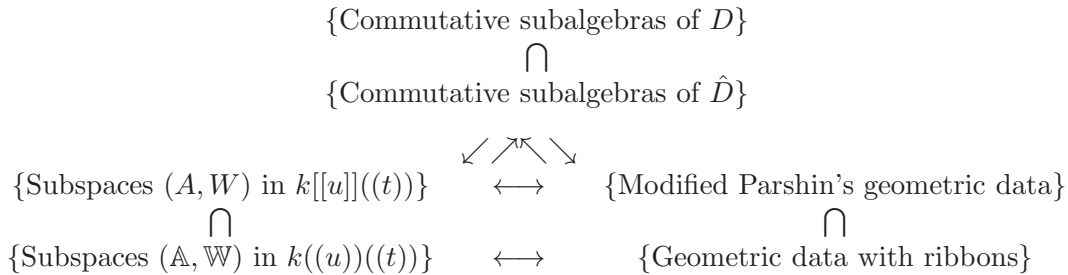
$$
\begin{array}{ccc}
& \{\mbox{Commutative subalgebras of $D$}\} &\\
& \bigcap & \\
& \{\mbox{Commutative subalgebras of $\hat{D}$}\} &\\
\end{array}
$$
$$
\begin{array}{ccc}
&\swarrow  \nearrow \nwarrow \searrow   &\\
\{\mbox{Subspaces $(A,W)$ in $k[[u]]((t))$}\} & \longleftrightarrow & \{\mbox{Modified Parshin's geometric data}\} \\
 \bigcap &&\bigcap  \\
\{\mbox{Subspaces $(\mathbb{A},\mathbb{W})$ in $k((u))((t))$}\} & \longleftrightarrow & \{\mbox{Geometric data with ribbons}\} \\
\end{array}
$$

\bigskip

Summing everything said above about two approaches we come to the following observation: starting from a commutative ring of PDOs (in any $n$ variables) one can naturally construct a data consisting of a projective variety (a completion of the spectral variety), the divisor at infinity and a coherent sheaf (not necessary of rank one) on this projective variety. We start our paper with a construction of this geometric data. This data corresponds to a commutative ring of PDO  satisfying certain conditions (see section \ref{sec2}).

After that we compare in details this construction and the construction given in \cite{Zhe2} in the case of operators in two variables (see section~\ref{sec3}).
 As a result of the comparison of two approaches we find necessary conditions on geometric data from \cite{Zhe2} describing commutative rings of PDOs. We also come to a problem (see problem \ref{problem1}) which is important in solving the problem of classification of commutative rings of PDOs (see section \ref{sec2.1.1})

 In section \ref{ribbons} we show how the construction from \cite{Ku}, \cite{Ku1} can be extended to associate with a geometric data from \cite{Zhe2} a geometric data from \cite{Ku}. So, in particular, one can apply the theory of ribbons developed in \cite{Ku}, \cite{Ku1} to study rings of PDOs and their isospectral deformations.

In section \ref{glueing} we recall the construction of glueing closed subschemes and give several examples that could be useful in solving the problem \ref{problem1}.

In appendix A we recall the construction and properties of the cycle map defined for any singular projective variety. It is used in a basis of almost all our geometric constructions.

In appendix~B we introduce a notion of Cohen-Macaulaysation
of a surface (cf. \cite[sec.3]{Bur}, \cite{Fa}). Namely, we show that for given integral
two-dimensional scheme $X$ of finite type over a field $k$ (or over
the integers) there is a "minimal" Cohen-Macaulay scheme $CM(X)$ and
a finite morphism $CM(X) \rightarrow X$ (and a finite morphism form
the normalization of $X$ to $CM(X)$). This construction generalizes
the known construction of normalisation of a scheme.
Using this construction we show
in section \ref{ribbons} (Prop. \ref{reconstruction}) that the image of the extended map applied
to a ribbon constructed by geometric data from \cite{Zhe2} coincides
with the image of the Parshin map applied to the
Cohen-Macaulaysation of this data. The equation \eqref{intersection} follows immediately from the proof of Prop. \ref{reconstruction}.

Everywhere we assume that a field $k$ has characteristic zero.

{\bf Acknowledgments.} Part of this research was done at the
Mathematisches Forschungsinstitut Oberwolfach during a stay within
the Research in Pairs Programme from January 23 till February~5,
2011. We would like to thank the MFO at Oberwolfach for the
excellent working conditions. We are also grateful to Igor Burban
and Andrey E. Mironov  for many stimulating discussions and useful
references.

We are grateful to the referee for the careful reading the article and suggestions of a lot of improvements for the exposition of our results.

The second author was partially supported by
Russian Foundation for Basic Research (grant no.~14-01-00178-a and no.~12-01-33024 mol\_a\_ved) and by
the Programme for the Support of Leading Scientific Schools of the
Russian Federation (grant no.~NSh-2998.2014.1). The third author was partially supported by the RFBR grant no.~14-01-00178-a, 13-01-00664à and by grant NSh no.~581.2014.1.

\section{Several constructions}
\label{sec2}

In this section we give the geometric construction mentioned in the introduction, i.e. we show how starting from a commutative ring of PDOs (in any $n$ variables) one can naturally construct a data consisting of a projective variety (a completion of the spectral variety), the divisor at infinity and a coherent sheaf (not necessary of rank one) on this projective variety.

We start this section with recalling some facts about rings of partial differential operators satisfying certain mild conditions (as it usually assumed in works about algebraically integrable systems) in sections \ref{generalities}, \ref{coordinates}, \ref{char}. The construction is given in section \ref{geomproperties}.

\subsection{Generalities}
\label{generalities}

Let $R$ be a commutative $k$-algebra, where $k$ is a field of characteristic zero.

Then we have
the $R$-module $\Der_k (R)$ of derivations and
the filtered ring $D(R)$ of $k$-linear differential operators.
By definition, the ring $D(R)$ is generated by $\Der_k(R)$ and $R$ inside the ring ${\End}_k(R)$.
The filtration has the properties:
$$
R = D_0(R)\subset D_1(R)\subset D_2(R)\subset \ldots \mbox{; } \quad D_i(R)D_j(R)\subset D_{i+j}(R) \mbox{; } \quad \Der\nolimits_k (R) = D_1(R) \mbox{.}
$$

The subspaces $D_i(R)$ are defined inductively as sub-$R$-bimodules of $\End_{k}(R)$. By definition, $D_0(R)=\End_{R}(R)=R$, and for $i \ge 0$
$$
D_{i+1}(R)=\{P\in D(R) \mid \mbox{ such that} \quad  [P,f]\in D_i (R) \quad   \mbox{for all} \quad  f\in R\}.
$$

Then we can form the graded ring
$$
gr (D(R))=\bigoplus_{i=0}^{\infty}D_i(R)/D_{i-1}(R)\mbox{,\quad where} \quad D_{-1}(R)=0 \mbox{,}
$$
and for $P\in D_i(R)$ the {\it principal symbol} $\sigma_i(P)=P \mod D_{i-1}(R)$. For $P\in D_i$, $Q\in D_j$ we have that $\sigma_i(P)\sigma_j(Q)=\sigma_{i+j}(PQ)$ and
$[P,Q]\in D_{i+j-1}(R)$. Hence $gr (D(R))$ is a commutative graded $R$-algebra with a Poisson bracket
$$\{\sigma_i(P), \sigma_j(Q)\}=\sigma_{i+j-1}([P,Q])$$
with the usual properties.

\begin{defin}
\label{defin7}
We denote the order function from $D(R)$ to non-negative integers as
$$
\Ord (P)=\inf \{n  \mid  P\in D_n(R)   \}   \mbox{.}
$$
\end{defin}

\subsection{Coordinates}
\label{coordinates}

\begin{defin} \label{defcoord}
We say that $R$ has a system of coordinates $(x_1,\ldots ,x_n)\in R^n$ if the following  condition is satisfied:
the map
$$
{\Der}_k (R) \stackrel{\phi}{\longrightarrow} R^n  \quad \mbox{:} \quad D\longmapsto (D(x_1), \ldots ,D(x_n))
$$
is bijective.
\end{defin}

In this case there are uniquely defined $\partial_1,\ldots ,\partial_n\in {\Der}_k(R)$ such that
$$
\partial_i(x_j)=\delta_{ij} \mbox{.}
%\qquad \Ker (\partial_1)\cap\ldots\cap\Ker (\partial_n)=k \mbox{.}
$$
Then $\Der (R)$ is a free $R$-module with generators $\partial_1,\ldots ,\partial_n$. Besides, we have $[\partial_i,\partial_j]=0$. One checks (by induction on the grade) that
$$
R[\xi_1, \ldots ,\xi_n] \simeq gr (D(R)) \quad  \mbox{by } \quad \xi_i \mapsto \, \partial_i \hspace{-0.1cm} \mod D_0(R) \, \in gr_1(D(R)) \mbox{.}
$$
Also for $P\in D_i(R)$, $Q\in D_j(R)$ we have
\begin{equation}  \label{pueq}
\{\sigma_i (P), \sigma_j(Q)\} = \sum_{v=1}^n \frac{\partial \sigma_i(P)}{\partial \xi_v}\partial_v(\sigma_j(Q))- \sum_{v=1}^n \frac{\partial \sigma_j(Q)}{\partial \xi_v} \partial_v(\sigma_i(P))
\end{equation}
(where we have extended $\partial_v$ to $R[\xi_1,\ldots ,\xi_n]$ by $\partial_v(\xi_l)=0$).

A typical example of a ring with a coordinate system is the ring $k[x_1,\ldots ,x_n]$ or $k[[x_1,\ldots ,x_n]]$, where in the last case we have to restrict ourself
to the $R$-module of continuous derivations and
to the ring of continuous differential operators with respect to the usual topology on
$k[[x_1,\ldots ,x_n]]$ given by the maximal ideal. The ring $k[[x_1,\ldots ,x_n]]$ will be important for the main part of the article.

If $(y_1,\ldots , y_n) \in R^n$ is another coordinate system, we get a new basis $(\partial_1',\ldots ,\partial_n')$ of ${\Der}_k(R)$. Hence the change of generators is given by the matrix
$$
\left(
\begin{array}{ccc}
\partial_1(y_1)&\ldots &\partial_n(y_1)\\
\partial_1(y_2)&\ldots &\partial_n(y_2)\\
\vdots &\ddots &\vdots \\
\partial_1(y_n)& \ldots &\partial_n(y_n)
\end{array}
\right)
=M \mbox{,}
$$
as $(\partial_1', \ldots ,\partial_n')M=(\partial_1, \ldots ,\partial_n)$,
$(\xi_1', \ldots ,\xi_n')M=(\xi_1, \ldots ,\xi_n)$.

\begin{nt} \label{tangent}
In definition~\ref{defcoord} the  partial derivatives are $\partial_i= \phi^{-1}(0, \ldots, 1, \ldots,0)$.
Therefore this definition implies the trivialization $\phi$ of the tangent sheaf of $\Spec R$ (or of the tangent sheaf of $\Spf R$ when $R = k[[x_1, \ldots, x_n]]$).
Moreover, to say that the ring $R$ has a system of coordinates $(x_1,\ldots ,x_n)\in R^n$ is equivalent to say that there is a map
$$
\psi \, : \, \Spec R \lto \Spec k[x_1, \ldots, x_n]
$$
such that the induced  map of $\co_{\Spec R}$-modules
$$
\phi \, : \,  {\mathcal T}_{\Spec R}  \lto \psi^* {\mathcal T}_{\Spec k[x_1, \ldots, x_n]}
$$
is an isomorphism. (Here ${\mathcal T}_{\Spec R}$ and ${\mathcal T}_{\Spec k[x_1, \ldots, x_n]}$ are the corresponding tangent sheaves).
\end{nt}

\subsection{Characteristic scheme}
\label{char}

If $J\subset D$ is a right ideal, then we obtain a homogeneous ideal $\langle \sigma_i(P), P\in J\rangle$ in $gr (D)$ and a subscheme defined by this ideal either in $\Spec (gr(D))$ or in $\Proj (gr(D))$. Both are called the characteristic subscheme $\Ch (J)$. We consider the characteristic subscheme in $\Proj (gr(D))$.

If we have and fix a coordinate system, then according to remark~\ref{tangent} we have a trivialization $\phi$ of the  tangent sheaf. This trivialization  induces the trivialization of the cotangent sheaf and of the projective cotangent bundle. Hence we obtain
$$\Proj (gr(D))=\Proj (R[\xi_1,\ldots ,\xi_n])=\Spec (R)\times_k\dpp_k^{n-1} \mbox{.}$$
Consider the case of the ideal $J=PD$, where $P$ is an operator with $\Ord (P)=m$. If $\sigma_m(P)\in k[\xi_1,\ldots ,\xi_n]$, then we say that {\it the principal symbol is constant.} In this case the characteristic scheme is essentially given by the divisor of zeros of $\sigma_m(P)$ in $\dpp^{n-1}_k$, we call it $\Ch_0(P)$. It is unchanged by a $k$-linear change of coordinates.

\begin{lemma}
\label{wellness}
If $P_1,\ldots P_n$ are operators with constant principal symbols (with respect to a coordinate system $(x_1,\ldots ,x_n)$) and if $\det (\partial \sigma (P_i)/\partial \xi_j)\neq 0$, then any operator $Q$ with $[P_i,Q]=0$, $i=1,\ldots , n$ has also a constant principal symbol.
\end{lemma}

\begin{proof}
Let $m_i=\Ord (P_i)$ and $m = \Ord (Q)$.
From equality~\eqref{pueq} we have
$$
0=\{\sigma_{m_i}(P_i), \sigma_{m}(Q)\} =\sum_{v=1}^n\frac{\partial \sigma_{m_i}(P_i)}{\partial \xi_v} \partial_v(\sigma_m(Q)).
$$
for $i=1,\ldots ,n$. Since $\det (\partial \sigma_{m_i} (P_i)/\partial \xi_j)\in k[\xi_1,\ldots ,\xi_n]$ is not zero, we infer that
$\partial_j(\sigma_m (Q))=0$  for $j=1,\ldots n$. Hence $Q$ has constant principal symbol with respect to $(x_1,\ldots ,x_n)$.
\end{proof}

For any subring $F \subset D$ we {\em define} a filtration on $F$ which is induced by filtration of $D$: $F_n = F \cap D_n =
\{f \in F \mid \Ord(f) \le n \}$. We define  the ring $\gr (F) = \bigoplus\limits_{n=0}^{\infty} F_n/F_{n-1}$.

\subsection{Geometric properties of commutative rings of PDOs}
\label{geomproperties}

To formulate the main theorem in this section we recall some facts from algebraic geometry.
For any $n$-dimensional irreducible projective variety $X$ over the field $k$, and any Cartier divisors $E_1, \ldots, E_n \in \Div(X)$ on $X$ one defines the intersection index $(E_1 \cdot \ldots \cdot E_n)  \in \dz$ on $X$ (see, e.g.,~\cite{Fu}, \cite[ch.~1.1]{La}.)
Let $(E^n)= (E \cdot \ldots \cdot E)$ be the self-intersection index of a Cartier divisor $E \in \Div(X)$ on $X$, and $\cf$ be a coherent sheaf
on $X$. There is the asymptotic Riemann-Roch theorem (see survey in~\cite[ch.~1.1.D]{La}) which says that  the Euler characteristic
$\chi(X, \cf \otimes_{\co_X} \co_X(mE))$ is a polynomial of degree $\le n$ in $m$, with
\begin{equation}  \label{rrf}
\chi(X, \cf \otimes_{\co_X} \co_X(mE)) = \rk(\cf) \cdot \frac{(E^n)}{n!} \cdot m^n + O(m^{n-1}) \mbox{,}
\end{equation}
where $\rk$ is the rank of sheaf.

There is the cycle map: ${\rm Z} : \Div(X) \to \WDiv(X)$ from the Cartier divisors to the Weil divisors on $X$ (see appendix~A).
From~\cite[Ch.~2]{Fu}
 it follows that if $E_1, E_2 \in \Div(X)$ such that ${\rm Z}(E_1) = {\rm Z}(E_2)$, then the self-intersection indices $(E_1^n)= (E_2^n)$ on $X$.

The cycle map $\rm Z$ restricted to the semigroup of effective Cartier divisors $\Div^+(X)$ is an injective map to the semigroup of effective Weil divisors $\WDiv^+(X)$ not contained in the singular locus. We {\em will say } that  an effective Weil divisor $C$ on $X$ not contained in the singular locus is a {\em $\dq$-Cartier} divisor on $X$ if $lC \in { \rm Im} \, ({\rm Z} \mid_{\Div^+(X)})$ for some integer $l >0$.

\begin{defin}
Let $C$ be a $\dq$-Cartier divisor on $X$. We define the self-intersection index $(C^n)$ on $X$ as
\begin{equation} \label{intin}
(C^n) = (G^n)/ l^n \mbox{,}
\end{equation}
where $G= lC$ is a Cartier divisor for some integer $l >0 $.
\end{defin}
We note that if $l >0$ is  minimal  such that $lC$ is a Cartier divisor, then for any other $l' > 0 $ with the property $l'C$ is a Cartier divisor
we have that  $l \mid l'$. Therefore, using above reasonings and the property  $(E_1^n)= m^n(E_2^n)$ for any $E_1= m E_2$, $ E_2 \in \Div(X)$, $m \in \dz$
we obtain that formula~\eqref{intin} does not depend on the choice of appropriate $l$.

\smallskip

\begin{theo}
\label{techn5.2}
 Let $P_1,\ldots ,P_n\in D=k[[x_1,\ldots ,x_n]][\partial_1,\ldots ,\partial_n]$ be any commuting operators of positive order.  Let $B$ be any commutative $k$-subalgebra in $D$ which contains the operators $P_1,\ldots ,P_n$.
Assume that
the intersection of the characteristic divisors of $P_1,\ldots ,P_n$ is empty.

Then the map from $gr(D)$ to $gr(D)/x_1gr(D) + ... + x_ngr(D) = k[ \xi_1, ..., \xi_n]$ induces an embedding on $gr(B)$ and we also have the following properties.
\begin{enumerate}
\item \label{fg}
 $ k[\xi_1,\ldots ,\xi_n]$ is finitely generated as $\gr (B)$-module.
\item
The rings $B$ and $\gr B$ are finitely generated integral $k$-algebras of Krull dimension $n$.
\item \label{uni}
The affine variety $U = \Spec \, B$ over $k$ can be naturally completed to an $n$-dimensional irreducible projective variety $X$ with boundary $C$ which is an integral Weil divisor not contained in the singular locus of $X$. Moreover, $C$ is an unirational and ample $\dq$-Cartier   divisor.
\item The $B$-module $L = D/  x_1 D + \ldots +  x_n D$, which defines a coherent sheaf on $U$, can be naturally extended to a torsion free coherent
sheaf $\cl$ on $X$. Moreover, the self-intersection index $(C^n)$ on $X$  is equal to $\delta^n/\rk (\cl)$, where 
\begin{equation}
\label{delta}
\delta= {\rm gcd} \ \{n   \mid  B_n/B_{n-1} \ne 0 , \, n\ge 1\} \mbox{.}
\end{equation}

\end{enumerate}
\end{theo}
\begin{proof}  We prove items $1$ and $2$. Let $m_i=\Ord (P_i)$. Denote by $\sigma_{m_i}'(P_i)$ the images of $\sigma_{m_i}(P_i)$ in $k[\xi_1,\ldots ,\xi_n]$.
 Now $(\sigma_{m_1}'(P_1),\ldots ,\sigma_{m_n}'(P_n)):\da^n\rightarrow \da^n$ is a finite morphism by Hilbert's Nullstellensatz, since the system of equations $\sigma_{m_1}'(P_1) = 0, \ldots, \sigma_{m_n}'(P_n)=0 $ defines only the zero point in $\da^n$ because of our assumption. In particular, $\det (\partial \sigma_{m_i}'(P_i)/\partial \xi_j)\neq 0$ (via the interpretation as the map on the tangent space) and therefore the matrix $(\partial \sigma_{m_i}(P_i)/\partial \xi_j)$ is invertible over the ring $k[[x_1,\ldots ,x_n]]((\xi_1^{-1}))\ldots ((\xi_n^{-1}))$.

Let's show that for any element $Q\in B$ the image $\sigma_{\Ord (Q)}' (Q)$ of $\sigma_{\Ord (Q)} (Q)$ in $k[\xi_1,\ldots ,\xi_n]$ is not zero, i.e. $gr (B)$ is embedded in $k[\xi_1,\ldots ,\xi_n]$ and $L$ is a torsion free $B$-module. Assume the converse, and let $N>0$ be the minimal value of the discrete valuation with respect to the maximal ideal in $k[[x_1,\ldots ,x_n]]$ on the coefficients of $\sigma_{\Ord (Q)} (Q)\in k[[x_1,\ldots ,x_n]][\xi_1,\ldots ,\xi_n]$. Since $[P_i,Q]=0$ for $i=1,\ldots ,n$, the equality $\{\sigma_{m_i}(P_i), \sigma_{\Ord (Q)} (Q)\}=0$ holds. On the other hand, from equality~\eqref{pueq} we have
$$
\{\sigma_{m_i}(P_i), \sigma_{\Ord (Q)} (Q)\}=\sum_{v=1}^n\frac{\partial \sigma_{m_i}(P_i)}{\partial \xi_v}\partial_v(\sigma_{\Ord (Q)} (Q)) \mod (x_1,\ldots ,x_n)^N
$$
Thus, we must have that $\partial_v(\sigma_{\Ord (Q)} (Q))=0 \mod (x_1,\ldots ,x_n)^N$ for all $v$, because the matrix $(\partial \sigma_{m_i}(P_i)/\partial \xi_j)$ is invertible also over the ring $T((\xi_1^{-1}))\ldots ((\xi_n^{-1}))$,
where $T=k[[x_1,\ldots ,x_n]]/ (x_1,\ldots ,x_n)^N$.
We obtained  a contradiction, since $N$ was chosen minimal.

Now we have
\begin{equation}  \label{emb}
k[\sigma_{m_1}'(P_1),\ldots ,\sigma_{m_n}'(P_n)]\subset \gr (B)\subset k[\xi_1,\ldots ,\xi_n].
\end{equation}
Hence $B_0 =k$.
But $k[\xi_1, \ldots ,\xi_n]$ is finitely generated  as $k[\sigma_{m_1}'(P_1),\ldots ,\sigma_{m_n}'(P_n)]$-module. Therefore the $k$-algebra $\gr B$ is a finitely generated $k$-algebra of Krull dimension $n$. Besides,  $k[\xi_1,\ldots ,\xi_n]$ is finitely generated as $\gr B$-module.
From~\eqref{emb}  it follows  that $\gr B$ is a ring without zero divisors. Hence the ring $B$ itself is without zero divisors.

 It will be useful to introduce the analog of the Rees ring $\tilde{B}$ constructed by the filtration on the ring $B$: $\tilde{B}=\bigoplus\limits_{n=0}^{\infty} B_n s^n$. The ring $\tilde{B}$ is a subring of the polynomial ring $B[s]$.
 For the fields of fractions we have $\Quot \tilde{B} = \Quot B[s]$.
 Besides, $\gr B=\tilde{B}/(s)$.
 Let
 the $k$-algebra $\gr(B)$ be generated by elements $\sigma_{m_i}(b_i)$, $i=1, \ldots, p$ as $k$-algebra,
 where $\Ord(b_i) = m_i$. It is easy to check that the $k$-algebra $B$ is generated by the elements $b_i$, $i=1, \ldots, p$ as  $k$-algebra,
 and the $k$-algebra $\tilde{B}$ is generated by the elements $s, b_1s^{m_1}, \ldots, b_rs^{m_p}$ as $k$-algebra.
  Hence we can compute   the Krull dimension of the ring $B$:
 \begin{equation}  \label{dim}
 \dim B=\trdeg \Quot B=\trdeg \Quot \tilde{B}-1=\trdeg \Quot (\tilde{B}/(s))=\trdeg \Quot (\gr B)=n  \mbox{,}
  \end{equation}
  since $(s)$ is a prime ideal of height $1$ in the ring $\tilde{B}$ by Krull's height theorem.

\smallskip

We prove now item $3$.
The ideal $I= \bigoplus\limits_{n=1}^{\infty} B_{n-1} s^n = (s) $ is a homogeneous  ideal in the ring $\tilde{B}$, because this ideal is generated by the homogeneous element $s \in \tilde{B}$. Besides, $I$ is a prime ideal, since  $\tilde{B}/I = \gr B$ is a ring without zero divisors.

We introduce the schemes $X= \Proj \, \tilde{B}$ and $C=\Proj \, \tilde{B}/I=\Proj  \, \gr (B)$.
Since $\tilde{B}$ and $\gr (B)$ are integral $k$-algebras, $X$ and $C$ are integral schemes.
Therefore, using~\eqref{dim}, we have that the homogeneous prime ideal $I$ defines an irreducible subscheme $C$ of codimension $1$ on $X$. Moreover,
$X \setminus C = \Spec \, \tilde{B}_{(s)} = \Spec B$ is an affine variety. (Here $\tilde{B}_{(s)}$ is the subring of degree zero elements in the localization $\tilde{B}_s$ of the ring $\tilde{B}$ by the multiplicative system $s^n$, $n \in \dz$).

For any $n \ge 0$ we denote  the homogeneous component $\tilde{B}_n= B_{n}s^n \subset \tilde{B}$.
Since $\tilde{B}$ is a finitely generated $k$-algebra with $\tilde{B}_0=k$, by~\cite[Ch.III, \S~1.3, prop.~3]{Bu} there exists an integer $d \ge 1$ such that the 
$k$-algebra $\tilde{B}^{(d)}=\bigoplus\limits_{k=0}^{\infty} \tilde{B}_{kd}$ is finitely generated by elements from $\tilde{B}_1^{(d)}$ as a graded $k$-algebra. (Here $\tilde{B}_1^{(d)}= \tilde{B}_d$, and $\dim_k \tilde{B}_1^{(d)} < \infty$ by formula~\eqref{emb}.) Therefore the scheme $\Proj \, \tilde{B}^{(d)}  \hookrightarrow \Proj \, {\rm Sym}_k (\tilde{B}_1^{(d)}) \simeq \dpp^N_k$ is a projective scheme over $k$ which is an irreducible variety. Clearly, $\delta |d$, since $\tilde{B}^{(d)}$ is a graded $k$-algebra. 

Let us show  that $(d/\delta )C$ is a very ample effective Cartier divisor on $X$.
We consider the subscheme $C'$ in $X$ which is  defined by the homogeneous ideal $I^d=(s^d)$ of the ring $\tilde{B}$.  The topological space of the subscheme $C'$ coincides with the topological space of the subscheme $C$ (as it can be seen on an affine covering of $X$). The function $-\ord / \delta : (\Quot B)^* \to \dz$ is a surjective function which defines the discrete valuation on the field $\Quot B$. The local ring $\co_{X,C}$ coincides with the valuation ring of this discrete valuation:
$$
\co_{X,C} = \tilde{B}_{(I)}= \{  a s^n / b s^n  \, \mid \,  n \ge 0,  a \in B_n, \,  b \in B_n \setminus B_{n-1} \} \mbox{.}
$$
The ideal $I^{\delta}$ induces the maximal ideal in the ring $\co_{X,C}$, and the ideal $I^{d}$ induces the $d/\delta$-th power of the maximal ideal.
 Therefore, if we will prove that the ideal $I^d$ defines an effective Cartier divisor on $X$, then the cycle map on this divisor is equal to $(d/\delta )C$ (see appendix~A), i.e. $C$ is a $\dq$-Cartier divisor.
 By~\cite[prop. 2.4.7]{EGAII} we have $X=\Proj \, \tilde{B}\simeq \Proj \, \tilde{B}^{(d)}$. Under this isomorphism the subscheme $C'$ is defined by the homogeneous ideal $I^d \cap \tilde{B}^{(d)}$ in the ring $\tilde{B}^{(d)}$. This ideal is generated by the element $s^d\in \tilde{B}_1^{(d)}$.
The open affine subsets $D_+(x_i)=\Spec \, \tilde{B}^{(d)}_{(x_i)}$ with $x_i\in \tilde{B}_1^{(d)}$ define a covering of $\Proj \, \tilde{B}^{(d)}$. In every ring $\tilde{B}^{(d)}_{(x_i)}$ the ideal $(I^{d} \cap \tilde{B}^{(d)})_{(x_i)}$ is generated by the element $s^d/x_i$. Therefore the homogeneous ideal
$I^{d} \cap \tilde{B}^{(d)}$ defines an effective Cartier divisor.

At last, the Cartier divisor $(d/\delta )C$ is a  very ample divisor, because $C'$ is a hyperplane section in the embedding $X=\Proj \, \tilde{B}^{(d)}\hookrightarrow \Proj \, {\rm Sym}_k (\tilde{B}_1^{(d)}) \simeq \dpp^N_k$. Besides, by item $1$, $k[\xi_1,\ldots \xi_n] \supset \gr B$, and $k[\xi_1,\ldots \xi_n]$ is a finite  $\gr (B)$-module. Hence the divisor $C=\Proj\,  \gr (B)$ is an unirational variety.

Since $\co_{X,C}$
is a regular local ring, the divisor $C$ is not contained in the singular locus of~$X$.

\smallskip

We prove now item $4$.
Let's define the sheaf $\cl$. Consider the right $D$-module $$L=D/(x_1D+\ldots +x_nD)$$
with filtration $L_n=(D_n+x_1D+\ldots +x_nD)/(x_1D+\ldots +x_nD)$. Then we have $L_nB_r\subseteq L_{r+n}$.

We consider another right $D$-module $k[\xi_1, \ldots, \xi_n]$ with the action of $D$ given (generated) as:
$$
f \circ \partial_j = f\xi_j    \qquad  \mbox{,} \qquad  g \circ x_i = - \frac{\partial g}{\partial \xi_i}
$$
for any $1 \le i,j \le n$, $f, g \in k[\xi_1, \ldots, \xi_n]$. It is easy to check that the maps
$$ k[\xi_1, \ldots, \xi_n]  \lto L \quad    : \quad
\sum_{\alpha \in \mathbb{N}^n} a_{\alpha} \xi^{\alpha}  \mapsto \sum_{\alpha} a_{\alpha}
  \partial^{\alpha}  $$
 $$ L \lto k[\xi_1, \ldots, \xi_n]  \quad    : \quad \sum_{\alpha \in \mathbb{N}^n}  p_{\alpha}(x) \partial^{\alpha} \mapsto \sum_{\alpha}
 p_{\alpha}(0)  \xi^{\alpha} \mbox{,} $$
 where $a_{\alpha} \in k$, $p_{\alpha} \in k[[x_1, \ldots, x_n]]$,
 $
 \xi^{\alpha}= \xi_1^{\alpha_1} \ldots \xi_n^{\alpha_n}$,
  $\partial^{\alpha}= \partial_1^{\alpha_1} \ldots \partial_n^{\alpha_n} $,
 are isomorphisms of the corresponding $D$-modules (and hence $B$-modules). The filtration on $k[\xi_1, \ldots, \xi_n]$ is the degree filtration
 of the polynomials. Therefore we have that for any integer $m \ge 0$
 \begin{equation} \label{fdim}
 \dim_k L_m ={m+n \choose n} =\frac{(m+1) \cdot \ldots \cdot (m+n)}{n!}  \mbox{.}
 \end{equation}
 Moreover,
 $\gr (L)= \bigoplus\limits_{n=1}^{\infty} L_n/L_{n-1} =k[\xi_1,\ldots ,\xi_n]$ is a finitely generated $\gr (B)$-module (see item~\ref{fg}).
 Now we have by induction on the degree of filtration that if elements $\sigma_{m_1}(v_1), \ldots, \sigma_{m_s}(v_q)$ (where $v_i \in L_{m_i}$,
 $\sigma_{m_i}(v_i)= v_i \mod L_{m_i-1}$, $i=1, \ldots, q$)  generate  $\gr(L)$ as a  $\gr(B)$-module, then elements $v_1, \ldots, v_q$ generate $L$ as $B$-module.
   Hence we obtain that $\tilde{L}=\bigoplus\limits_{m=0}^{\infty}L_m s^m$ is a finitely generated torsion free graded $\tilde{B}$-module which is generated
    by elements $v_1 s^{m_1}, \ldots , v_q s^{m_q}$ over the ring $\tilde{B}$.
    Therefore  $\cl =\Proj \tilde{L}$ is a torsion free coherent sheaf\footnote{Here and later in the article we use the non-standard notation $\Proj $ for the quasi-coherent sheaf associated with a graded module. If $M$ is a  filtered module, then we use the notation $\tilde{M}=\bigoplus\limits_{i=0}^{\infty} M_is^i$ for the analog of the Rees module, as well as for  filtered rings.} on $X$
    (see~\cite[prop.~2.7.3]{EGAII}).
    Besides, the graded $\gr B$-module $\gr L$ defines the torsion free coherent sheaf over $C= \Proj \, \gr B$, and the $B= \tilde{B}_{(s)}$-module
    $L=\tilde{L}_{(s)}$ defines the torsion free coherent sheaf over $X \setminus C = \Spec B$.

    We have $X = \Proj \tilde{B}^{(d)}$. Under this isomorphism the graded $\tilde{B}^{(d)}$-module
    $\tilde{L}^{(d)}=\bigoplus\limits_{k=0}^{\infty} \tilde{L}_{kd} $ (where $\tilde{L}_{kd}= L_{kd}s^{kd}$) gives the coherent sheaf $\cl$
    as $\Proj \tilde{L}^{(d)}$. We have proved that $C'=(d/\delta )C$ is a very ample Cartier divisor on the projective variety $X$.
    Therefore, by~\cite[ch.~III, th.~5.2]{Ha},
    $$
    H^i (X, \cl \otimes_{\co_X} \co_X(mC')) =0  \quad \mbox{for} \quad i>0 \quad \mbox{and} \quad m \gg 0 \mbox{.}
    $$
    Also, by~\cite[ch.~II, exerc.~5.9(b)]{Ha}, $H^0 (X, \cl \otimes_{\co_X} \co_X(mC')) = \tilde{L}_{md} $ for $m \gg 0$. Hence and from
    formula~\eqref{fdim} we obtain
    $$
    \chi(X, \cl \otimes_{\co_X} \co_X(mC')) = \frac{(md+1) \cdot \ldots \cdot (md+n)}{n!} \quad \mbox{for} \quad m \gg 0 \mbox{.}
    $$
    From formula~\eqref{rrf} we have that the self-intersection index $({C}'^n)=d^n/\rk(\cl)$ on $X$. Hence, the self-intersection index
    $(C^n) = \delta^n/ \rk(\cl)$ on $X$.
    \end{proof}

\begin{nt}
The items $1$ and partially item $2$ of theorem~\ref{techn5.2} follow also from \cite[Ch.III, \S 2.9, Prop.~10]{Bu}. The item $2$ was proved in~\cite{Kr} by Krichever in connection with integrable systems. We gave here an alternative proof in the spirit of pure commutative algebra.

The sheaf $\cl$ is a Krichever sheaf in the sense of \cite[introduction]{ZhM}.
It is in some sense similar to the sheaf from family of Krichever sheaves (or Baker-Akhieser modules), confer \cite{ZhM}.
\end{nt}

\begin{nt}
The sheaf $\cl$ determines the bundle of eigenfunctions of operators of $B$ on an open part of $X$ as follows (we assume $k$ is algebraically closed). For $U=\Spec (B)$ the points $p\in U(k)$ correspond one to one to the characters $\chi_p$ of the $k$-algebra $B$ (i.e. $k$-algebra morphisms $\chi : B \rightarrow k$ given by maximal ideals of the ring $B$). To the coherent sheaf $\cl$ we have the associated affine scheme $\Specz (\cl )$ together with the morphism $\pi : \Specz (\cl ) \rightarrow X$, see \cite[ex.5.17, ch. II]{Ha} (we'll call it the associated bundle, it is really a vector bundle over an open part of $X$, where $\cl$ is locally free). Recall that this scheme is constructed by glueing schemes $\Spec (Sym(\cl (V)))$ for open affine sets $V$.
In particular, for $p\in U$ we have
$$
\pi^{-1} (p)={\Hom}_k (\cl (p),k)={\Hom}_B(L,k(p))
$$
(here $\cl (p)$ is the fiber of the coherent sheaf $\cl$ at the point $p$, i.e. $\cl (p) = \cl_p \otimes_{\co_{X,p}} k(p)$). These spaces are naturally isomorphic to the spaces
$$
V(\chi_p)=\{f\in R \mbox{,} \ P(f)=\chi_p(P)f \mbox{\quad for all $P\in B$} \}
$$
as follows:

(i) we have $R\simeq \Hom_{k}(L,k)$ (as vector spaces) by
$$
f\mapsto \lambda_f, \mbox{\quad} \lambda_f(P)=P(f)(0)
$$
$$
\lambda \mapsto f_{\lambda} =\sum_{v}\frac{1}{v !}\lambda (\partial^{v})x^{v}
$$
with $v =(v_1,\ldots ,v_n)\in \dn^n$, $v!=v_1!v_2!\ldots v_n!$ and $\partial^v=\partial_1^{v_1}\ldots \partial_n^{v_n}$, $x^v=x_1^{v_1}\ldots x_n^{v_n}$.

(ii) fixing a point $p\in U(k)$, the field $k=k(p)$ gets a $B$-module structure, and the isomorphism (i) gives
$$
V(\chi_p)\simeq {\Hom}_{B}(L, k(p))\subset {\Hom}_{k}(L,k).
$$
Thus, the bundle $\Specz (\cl )|U$ is the bundle of eigenfunctions of $B$, and we can consider $X$, $\Specz (\cl )$ as prolongation of the spectrum and of the bundle of eigenfunctions to infinity.
\end{nt}

\section{Operators in two variables}
\label{sec3}

In this section we work with the ring $D=k[[x_1,x_2]][\partial_1,\partial_2]$ of PDOs in two variables.

We deduce several properties of geometric data which classify the subrings in $\hat{D}$ and of those data which correspond to subrings in $D$, then establish a connection of any data with geometric data consisting of ribbons and torsion free sheaves introduced in \cite{Ku},\cite{Ku1}. At the end we write about the glueing construction of subschemes in a scheme and give some examples.

We start with recalling some definitions, theorems and constructions from \cite{Zhe2} (see sections \ref{catgemdat}, \ref{mapxi}, \ref{rings}). All results about various geometric properties are contained in section \ref{sec2.1.1}, the connection of geometric data with ribbon's data is given in section \ref{ribbons}, the glueing construction is given in section \ref{glueing}.

\subsection{Category of geometric data}
\label{catgemdat}

We recall that the ring $k[[u,t]]$ has a natural linear topology, where the base of neighbourhoods of zero is generated by the powers of the maximal ideal of this ring.

On the two-dimensional local field $k((u))((t))$ we will consider the following discrete valuation of rank two
$\nu \, : \, k((u))((t))^* \to \dz \oplus \dz $:
$$
\nu (f) = (m,l)  \quad \mbox{iff} \quad f= t^lu^m f_0 \mbox{, where} \quad f_0 \in k[[u]]^*+t k((u))[[t]] \mbox{.}
$$
(Here $k[[u]]^*$ means the set of invertible elements in the ring $k[[u]]$.)

We recall several definitions and results from  \cite{Zhe2}.
\begin{defin}
\label{geomdata}
We call $(X,C,P,\cf ,\pi , \phi )$ a geometric data of rank $r$ if it consists of the following data:
\begin{enumerate}
\item\label{dat1}
$X$ is a reduced irreducible projective algebraic surface defined over a field $k$;
\item\label{dat2}
$C$ is a reduced irreducible ample $\dq$-Cartier  divisor on $X$;
\item\label{dat3}
$P\in C$ is a closed $k$-point, which is
regular on $C$ and on $X$;
\item\label{dat4}
$$
\pi : \widehat{\co}_{P}\longrightarrow k[[u,t]]
$$
is a local $k$-algebra homomorphism satisfying the following
property. If $f$ is a local equation of the curve $C$ at $P$, then
$\pi (f)k[[u,t]] = t^rk[[u,t]]$ and the induced map $\pi :
\co_{C,P}=\co_P/(f) \rightarrow k[[u]] = k[[u,t]]/(t)$ is an
isomorphism. (The definition of $\pi$ does not depend on the choice
of appropriate $f$. Besides, from this definition it follows that
$\pi$ is an embedding, $k[[u,t]]$ is a free $\widehat{\co}_P$-module
of rank $r$ with respect to $\pi$. Moreover for any element $g$ from
the maximal ideal $\cm_P$ of $\co_P$ such that elements $g$ and $f$
generate $\cm_P$ we obtain that $\nu(\pi(f))=(0,r)$,
$\nu(\pi(g))=(1,0)$.)

\item\label{dat5}
$\cf$ is a torsion free quasi-coherent sheaf on $X$.
\item\label{dat6}
$\phi :{\cf}_P \hookrightarrow k[[u,t]]$ is an ${\co}_P$-module embedding
subject to the following condition for any $n \ge 0$ (we note that by item~\ref{dat4} of this definition, $k[[u,t]]$ is an ${\co}_P$-module with respect to $\pi$).
By item~\ref{dat2} there is the minimal natural number $d$ such that $C'=dC$ is a very ample divisor on $X$. Let $\gamma_n : H^0(X, \cf (nC'))\hookrightarrow {\cf}(nC')_P$ be an embedding (which is an embedding, since $\cf (nC')$ is a torsion free quasi-coherent sheaf on $X$).
Let $\epsilon_n : {\cf}(nC')_P \to \cf_P$ be the natural ${\co}_P$-module isomorphism  given by multiplication to an element $f^{nd} \in {\co}_{P}$, where $f \in {\co}_{P}$ is chosen as in item~\ref{dat4}. Let $\tau_n : k[[u,t]] \rightarrow k[[u,t]]/(u,t)^{ndr+1}$ be the natural
ring epimorphism. We demand that the map
$$    \tau_n \circ \phi \circ  \epsilon_n \circ \gamma_n \, : \, H^0(X, \cf (nC'))   \lto  k[[u,t]]/(u,t)^{ndr+1}$$
is an isomorphism. (These conditions on the map $\phi$ do not depend on the choice of the appropriate element $f$.)
\end{enumerate}
Two geometric data $(X,C,P,\cf ,\pi_1 , \phi_1 )$ and $(X,C,P,\cf ,\pi_2 , \phi_2 )$ are identified if the images of the embeddings
(obtained by means of multiplication to $f^{nd}$ as above)
$$
H^0(X, \cf (nC'))\hookrightarrow {\cf}_P\stackrel{\phi_1}{\hookrightarrow} k[[u,t]], \mbox{\quad} H^0(X, \co (nC'))\hookrightarrow \widehat{\co}_P\stackrel{\pi_1}{\hookrightarrow} k[[u,t]]
$$
and
$$
H^0(X, \cf (nC'))\hookrightarrow {\cf}_P\stackrel{\phi_2}{\hookrightarrow} k[[u,t]], \mbox{\quad} H^0(X, \co (nC'))\hookrightarrow \widehat{\co}_P\stackrel{\pi_2}{\hookrightarrow} k[[u,t]]
$$
coincide for any $n \ge 0$.
The set of all quintets of rank $r$ is denoted by $\cq_r$.
\end{defin}

\begin{nt}
We would like to emphasize that the rank $r$ of the geometric data in general differs from the rank of the sheaf $\cf$, see remark \ref{primer}.

If $\cf_P$ is a free $\co_P$-module of rank $r$, then $\phi$ induces an isomorphism $\widehat{\cf}_P\simeq k[[u,t]]$ of $\widehat{\co}_P$-modules.
This condition is satisfied if $\cf$ is a coherent sheaf of rank $r$, see corollary \ref{freemod} below.
\end{nt}

\begin{defin}
\label{geomcategory}
We define a category $\cq$ of geometric data as follows:
\begin{enumerate}
\item\label{cat1}
The set of objects is defined by
$$
Ob (\cq )=\bigcup_{r\in \sdn} \cq_r.
$$
\item\label{cat2}
A morphism
$$
(\beta , \psi ) \, : \, (X_1,C_1,P_1,\cf_1 ,\pi_1 , \phi_1 ) \longrightarrow (X_2,C_2,P_2,\cf_2 ,\pi_2 , \phi_2 )
$$
of two objects
consists of a morphism $\beta :X_1\rightarrow X_2$ of surfaces and a homomorphism $\psi :\cf_2\rightarrow \beta_*\cf_1$ of sheaves on $X_2$ such that:
\begin{enumerate}
\item\label{cat1.1}
$\beta |_{C_1} :C_1\rightarrow C_2$ is a morphism of curves;
\item\label{cat1.2}
$$
\beta (P_1)=P_2.
$$
\item\label{cat1.3}
There exists a continuous ring isomorphism $h:k[[u,t]] \rightarrow k[[u,t]]$ such that
$$
h(u)=u \mbox{\quad mod \quad} (u^2)+(t), \mbox{\quad} h(t)=t \mbox{\quad mod \quad} (ut)+(t^2),
$$
and the following commutative diagram holds:
$$
\begin{diagram}
\node{k[[u,t]]}   \arrow{e,t}{h} \node{k[[u,t]]}  \\
\node{\widehat{\co}_{X_2, P_2}}  \arrow{n,l}{\pi_2}  \arrow{e,b}{\beta_{P_1}^{\sharp}}
\node{\widehat{\co}_{X_1, P_1}}   \arrow{n,r}{\pi_1}
\end{diagram}
$$
\item\label{cat1.4}
Let's denote by ${\beta}_*(\phi_1)$ a  composition of morphisms of ${\co}_{P_2}$-modules
$${\beta}_*(\phi_1): {\beta_*\cf_1}_{P_2} \rightarrow {\cf_{1}}_{P_1}\hookrightarrow k[[u,t]].$$
There is a $k[[u,t]]$-module isomorphism $\xi :k[[u,t]] \simeq h_*(k[[u,t]])$  such that the following commutative diagram of morphisms of ${\co}_{P_2}$-modules holds:
$$
\begin{diagram}
\node{{\cf_2}_{P_2}} \arrow{e,t}{\psi}
\arrow{s,l}{\phi_2}
\node{{\beta_*\cf_1}_{P_2}}   \arrow{s,r}{{\beta}_*(\phi_1)} \\
\node{k[[u,t]]}  \arrow{e,b}{\xi} \node{h_*(k[[u,t]])=k[[u,t]]}
\end{diagram}
$$
\end{enumerate}
\end{enumerate}
\end{defin}

\subsection{Modified Schur pairs}
\label{mapxi}

Given a geometric data $(X,C,P,\cf ,\pi , \phi )$  (see section~\ref{catgemdat}) we will construct a pair of $k$-subspaces
$$W,A\subset k[[u]]((t))$$
such that $ A W \subset W$ and $A$ is a $k$-algebra  with unity (cf. also~\cite[Def.3.15]{Zhe2}).
These subspaces $(A,W)$ are the modified Schur pairs mentioned in introduction\footnote{The spaces $(A,W)$ were called  Schur pairs in \cite[Def.3.12]{Zhe2}.}.

Let $f^d$ be a local generator of the ideal $\co_X(-C')_P$, where $C'=dC$ is a very ample Cartier divisor (see also definition  \ref{geomdata}, item \ref{dat6}). Then $\nu (\pi (f^d))=(0,r^d)$ in the ring $k[[u,t]]$ and therefore  $\pi (f^d)^{-1}\in k[[u]]((t))$. So, we have natural embeddings for any $n >0$
$$
H^0(X, \cf (nC'))\hookrightarrow {\cf (nC')}_P\simeq f^{-nd} ({\cf }_P) \hookrightarrow k[[u]]((t)) \mbox{,}
$$
where the last embedding is the embedding $f^{-nd}{\cf }_P \stackrel{\phi }{\hookrightarrow } f^{-nd} k[[u,t]] {\hookrightarrow} k[[u]]((t))$ (see also definition \ref{geomdata}, item~\ref{dat6}). Hence we have the embedding
$$
\chi_1 \; : \; H^0(X\backslash C, \cf )\simeq \limind_{n >0} H^0(X, \cf (nC')) \hookrightarrow k[[u]]((t)) \mbox{.}
$$
We define now $$W \eqdef \chi_1(H^0(X\backslash C, \cf )) \mbox{.}$$
Analogously the embedding $H^0(X\backslash C, \co )\hookrightarrow k[[u]]((t))$ is defined (and we will denote it also by $\chi_1$). We define also
$$A \eqdef \chi_1(H^0(X\backslash C, \co )) \mbox{.} $$

As it follows from this construction,
\begin{equation}
\label{qwerty}
A\subset k[[u']]((t')) \subset k[[u]]((t)),
\end{equation}
where $t'=\pi (f)$, $u'=\pi (g)$ (see also definition \ref{geomdata}, item \ref{dat4}). Thus, on $A$ there is a filtration $A_n$ induced by the filtration ${t'}^{-n}k[[u']][[t']]$ on the space $k[[u']]((t'))$:
$$
A_n = {A \cap {t'}^{-n}k[[u']][[t']]}= {A \cap {t}^{-nr}k[[u]][[t]]}  \mbox{.}
$$
We have $X\simeq \Proj (\tilde{A})$, where $\tilde{A}=\bigoplus\limits_{n=0}^{\infty}A_n s^n$ (see also~\cite[lemma 3.3, th.3.3]{Zhe2}). The similar filtration is defined on the space $W\subset k[[u]]((t))$:
$$
W_n = {W \cap {t}^{-nr}k[[u]][[t]]}
$$
 And the sheaf $\cf \simeq \Proj (\tilde{W})$, where $\tilde{W}=\bigoplus\limits_{n=0}^{\infty}W_n s^n$. Note that we have $W_{nd}\simeq H^0(X, \cf (nC'))$ by definition \ref{geomdata}, item 6 and by construction of the map $\chi_1$.

\bigskip

We will also use in our paper  the  following notation. Let $\tilde{A}(i)$ (or $\tilde{W}(i)$) denote the graded ring (module) obtained from $\tilde{A}$ ($\tilde{W}$) by the shift of grading, i.e. the $k$-th homogeneous component is $\tilde{A}(i)_k=\tilde{A}_{k+i}$. Define the family of coherent sheaves on $X$:
$$
\cb_i=\Proj (\tilde{A}(i)), \mbox{\quad} \cf_i=\Proj (\tilde{W}(i)).
$$
Then from \eqref{qwerty} it easily follows that the sheaves $\cb_i/\cb_{i-1}$, $\cf_i/\cf_{i-1}$ are torsion free coherent sheaves on
$C\simeq \Proj (\cb_0/\cb_{-1})$.

\subsection{Commutative rings of operators}
\label{rings}

In this section we recall some definitions from \cite{Zhe2} needed to remind the classification of commutative $k$-algebras in the ring $\hat{D}$. We also recall several properties of commutative subalgebras in $D$ given in loc.cit.

Let's consider commutative $k$-algebras of PDOs $B\subset D=k[[x_1,x_2]][\partial_1,\partial_2]$ that satisfy the following condition (cf. conditions in theorem \ref{techn5.2}):
\begin{multline}
\label{property}
\mbox{$B$ contains the operators $P_1,P_2$ with constant principal symbols such that}\\
\mbox{ the intersection of the characteristic divisors of $P_1,P_2$ is empty. }
\end{multline}

In the work \cite{Zhe2} was shown that such algebras are a part of a wider set of commutative $k$-algebras $B'\subset \hat{D}$, and all algebras from this set can be classified in terms of geometric data from subsection \ref{catgemdat}.

\begin{defin}
Define
\begin{multline}
\hat{D}_1=\{a=\sum_{q\ge 0} a_{q}\partial_1^q\mbox{\quad }|  a_q\in k[[x_1, x_2]] \mbox{ and for any $N\in \dn$ there exists $n\in \dn$ such that }\\
\mbox{$\ord_{M}(a_m)>N$ for any $m\ge n$}\},
\end{multline}
where we define the number $\ord_M(a)$, $a\in k[[x_1, x_2]]$ as
$$
\ord_M(a)=\min\{k | a\in M^k\}, \mbox{\quad where $M=(x_1,x_2)$ is the maximal ideal in $k[[x_1, x_2]]$.}
$$

Define
$$
\hat{D}=\hat{D}_1[\partial_2].
$$
\end{defin}

The commutative subalgebras in $\hat{D}$ which can be classified in terms of modified Schur pairs or geometric data were called in \cite{Zhe2} as $1$-quasi elliptic strongly admissible rings. The exact definitions of such rings are too technical to give them here, because we will not use them in this paper (see definitions 2.18, 3.4, 2.11 in \cite{Zhe2}). In fact, we will only need to know the following.

If we have a ring $B\subset D$ of commuting PDOs satisfying the property \ref{property}, then by \cite[Lemma 2.6]{Zhe2} and by \cite[Prop.2.4]{Zhe2} (cf. also the beginning of section 3.1 in loc. cit.) there is a linear change of variables making this ring $1$-quasi elliptic strongly admissible.
Moreover, as it follows from the proofs of \cite[Lemma 2.6, Prop.2.4]{Zhe2}, almost all linear changes of variables preserve the property of the ring to be $1$-quasi elliptic strongly admissible.
From the construction of geometric data starting from $1$-quasi elliptic strongly admissible ring given in \cite[Sec.3]{Zhe2} it follows that the ring after such linear change of variables corresponds to the data with the same surface, divisor and sheaf, but with other point $P$ and trivializations $\pi ,\phi$.

There is another nice property of $1$-quasi elliptic subrings of PDOs claiming the "purity" of such rings:
\begin{prop}{(\cite[Prop.3.1]{Zhe2})}
\label{purity}
Let $B\subset D\subset \hat{D}$ be a $1$-quasi elliptic ring of commuting partial differential operators. Then any ring $B'\subset \hat{D}$ of commuting operators such that $B'\supset B$ is a ring of partial differential operators, i.e. $B'\subset D$.
\end{prop}

To recall the main theorem in this section we need a little bit more definitions:

\begin{defin}
The commutative $1$-quasi elliptic rings $B_1$, $B_2\subset \hat{D}$ are said to be equivalent if there is an invertible operator $S\in \hat{D}_1$ of the form $S=f+S^-$, where $S^-\in \hat{D}_1\partial_1$, $f\in k[[x_1, x_2]]^*$,  such that $B_1=SB_2S^{-1}$.
\end{defin}

The operators (differential or pseudo-differential) with constant coefficients are identified with elements of the space $k[z_1^{-1}]((z_2))$ via the replacements $z_1^{-1}\leftrightarrow \partial_1$, $z_2^{-1}\leftrightarrow \partial_2$.

The following $k$-linear maps give the connection between subspaces in $k[[u]]((t))$ and in $k[z_1^{-1}]((z_2))$:
$$
\psi : k[[u]]((t)) \rightarrow k[z_1^{-1}]((z_2)) \mbox{\quad } t\mapsto z_2, u\mapsto z_1^{-1}z_2,
$$
\begin{equation}
\label{psi_1}
\psi_1:\psi (k[[u]]((t)))\simeq k[[u]]((t)) \mbox{\quad } z_2\mapsto t, z_1^{-1}\mapsto ut^{-1}.
\end{equation}

\begin{theo}{(\cite[Th.3.4]{Zhe2})}
\label{dannye2}
There is a one to one correspondence between the set of classes of equivalent $1$-quasi elliptic strongly admissible finitely generated rings of operators in $\hat{D}$ and the set $\cm$ of isomorphism classes of geometric data (see definitions \ref{geomdata}, \ref{geomcategory}).
\end{theo}

The proof of this theorem is constructive: the surface and divisor from theorem  are constructed by the graded ring $\tilde{A}$ which is defined by the ring $A\subset k[[u]]((t))$, the sheaf and trivialisations are defined by the space $W$ (cf. section~\ref{mapxi} and see proof of theorem~3.3 in \cite{Zhe2}).
The ring $A$ and the subspace $W$ are defined by the ring $SBS^{-1}\subset k[z_1^{-1}]((z_2))$ and the space $W_0S^{-1}=(DS^{-1} \mod (x_1, x_2))\subset k[z_1^{-1}]((z_2))$ (where $W_0=k[z_1^{-1}, z_2^{-1}]$) and by change of variables $\psi_1$: $A=\psi_1(SBS^{-1})$, $W=\psi_1(W_0S^{-1})$, where $S$ is an invertible zeroth order pseudo-differential operator defined in \cite[lemma 2.11]{Zhe2}. Conversely, to construct the ring $B$ starting from the geometric data, one need to find the pair $(A,W)$ described in section \ref{mapxi} and then apply the map $\psi$ and the analogue of the Sato theorem \cite[Th.3.1]{Zhe2} to obtain the uniquely defined invertible zeroth order pseudo-differential operator $S$ such that $W_0S^{-1}=\psi (W) $. Then $B=S^{-1}\psi (A)S$.

\subsection{Algebro-geometric properties of geometric data}
\label{sec2.1.1}

Recall that rings of commuting partial differential operators with constant higher symbols give examples of $1$-quasi elliptic strongly admissible rings after appropriate change of variables (see section 3.1 and lemma 2.6 in \cite{Zhe2}).
Below we investigate several properties of surfaces and sheaves from geometric data in definition \ref{geomdata} in general case and in the case when the data comes from a ring of PDOs. Then we compare geometric properties of data obtained from a commutative algebra of PDOs in theorem \ref{techn5.2} and of data obtained from such data in theorem \ref{dannye2}. As a by-product we obtain some necessary geometric conditions on those data from definition \ref{geomdata} which describe the commutative algebras of PDOs.

First, in theorem~\ref{CM} and corollary~\ref{freemod} we investigate the Cohen-Macaulay properties of a surface $X$ and of a sheaf $\cf$ from geometric data introduced in definition~\ref{geomdata}. One of the application of these properties (see corollary~\ref{freemod}) is that the localization $\cf_P$  is a free $\co_P$-module for a coherent sheaf $\cf$ on $X$. (We note that we don't demand in definition~\ref{geomdata} the property that $\cf_P$ is a free  $\co_P$-module.) Another application of the Cohen-Macaulay property of a surface $X$ along a curve $C$ (see theorem~\ref{CM})
will be given later in section~\ref{ribbons}   in connection with the theory of ribbons (see~\cite{Ku,Ku1}). In particular, due to theorem~\ref{CM}, ribbons constructed by a geometric data and by a geometric data with the Cohen-Macaulaysation of the initial surface and of the initial sheaf are the same. It is an important property, because in this case, by analogy with the reconstruction theorems from~\cite{Os} and~\cite{Pa}, we will have proposition~\ref{reconstruction}  which allows us to reconstruct the Cohen-Macaulay surface $X$ (or the Cohen-Macaulaysation of the initial surface) from its ribbon.
(In papers~\cite{Os} and~\cite{Pa} the Cohen-Macaulay property of a surface from geometric data was important to reconstruct the  surface from its image in $k((u))((t))$.)

\begin{theo}
\label{CM}
Let $X$, $C$ be a surface and a divisor from geometric data in definition~\ref{geomdata}.
Then $X$ is Cohen-Macaulay outside a finite set of points
disjunct from $C$.
\end{theo}
\begin{proof} If we have a geometric data from definition \ref{geomdata}, we can define a ring $A\subset k[[u']]((t'))$ (see section~\ref{mapxi} above or \cite[def.3.15]{Zhe2}, \cite[th.3.3]{Zhe2}), a filtration $A_i$ defined by the discrete valuation $\nu_{t'}$ on the field $k((u'))((t'))$: $$A_i=\{a\in A \mid \nu_{t'}(a)\ge -i\} \mbox{,} \quad i \ge 0$$
which satisfy the following property: $A_{di}\simeq H^0(X,\co_X(idC))$ for all $i\ge 0$, where $d$ is a minimal natural number such that $dC$ is a very ample Cartier divisor. So, $X\simeq \Proj \tilde{A}^{(d)}\simeq \Proj \tilde{A}$, where
$\tilde{A}=\bigoplus\limits_{i=0}^{\infty}A_{i}s^i$, and $C$ is defined  by the homogeneous ideal $I=\tilde{A}(-1) =(s)$ in the ring $\tilde{A}$.
We note that
the ring $\tilde{A}$ is finitely generated over $k$, since the ring $\tilde{A}^{(d)} = \bigoplus\limits_{i=0}^{\infty} A_{di}s^{di}$ is finitely generated over $k$ as a graded ring which is equivalent to the homogeneous coordinate ring of the projective surface $X$, and the modules $\tilde{A}^{(d,l)} = \bigoplus\limits_{i=0}^{\infty} A_{di+l}s^{di+l}$, $0<l<d$ are naturally isomorphic to the ideals in $\tilde{A}^{(d)}$, which are finitely generated.  The curve $C$ is covered by all affine subsets $\Spec \tilde{A}_{(x_i)}$, where $x_i\in A_d$, $\nu_{t'}(x_i)=-d$.
 (The ring $\tilde{A}_{(x_i)}$ is the subring of homogeneous elements in the localization ring of $\tilde{A}$ with respect to the multiplicatively closed system $\{ x_i^m\}_{m \in \mathbb{Z}}$.)

First let's show that each point on the curve $C$ is Cohen-Macaulay on $X$.  It is enough to show that the ideal $(s^d/x_i)$ in the ring $\tilde{A}_{(x_i)}$ is $I_{(x_i)}$-primary for all $x_i$ with the above properties. (The element $s^d/x_i$ is non-zero divisor in the ring
$\tilde{A}_{(x_i)}$.
Therefore to prove the Cohen-Macaulay property it is enough to find a non-invertible non-zero divisor in the ring $\tilde{A}_{(x_i)}/ (s^d/x_i)$. But all zero-divisors in the last ring coincide with $I_{x_i}/ (s^d/x_i)$ if $(s^d/x_i)$ is a $ I_{(x_i)}$-primary ideal, and all these zero-divisors are nilpotent. Then by  \cite[prop.~4.7]{AM} and by Krull's theorem $\heit I_{(x_i)}=1$. Thus, $\dim \tilde{A}_{(x_i)}/I_{(x_i)}=1$ and there exist non-invertible non-zero divisors). Assume that elements $as^{dk}/x_i^k$ and $ bs^{dl}/x_i^l $ are from $\tilde{A}_{(x_i)}$, but not from the ideal $ (s^d/x_i)$, and
$$\frac{as^{dk}}{x_i^k} \cdot \frac{bs^{dl}}{x_i^l}=\frac{cs^{(k+l-1)d}}{x_i^{k+l-1}} \cdot \frac{s^d}{x_i} \in \left (\frac{s^d}{x_i}\right) \subset I_{(x_i)} \mbox{.}$$
We must show that $as^{dk}/x_i^k, bs^{dl}/x_i^l\in I_{(x_i)}$. Since $I_{(x_i)}$ is a prime ideal, without loss of generality we can assume that  the element $g=as^{dk}/x_i^k \in I_{(x_i)}$. Note that any element $y\in I_{(x_i)}$ satisfies the property $\nu_{t'}(y)>0$.
Then we have $\nu_{t'}(a)=-kd+j$, where $0<j<d$, because $g \in I_{(x_i)}$ and $g \notin (s^d/x_i)$ (if $j\ge d$, then $\nu_{t'}(a)\le (k-1)d$ and therefore $a\in A_{(k-1)d}\subset A_{kd}$, thus $as^{dk}=(as^{d(k-1)})s^d$ and $as^{dk}/x_i^k\in (s^d/x_i)$, a contradiction). Then  we have
 $$g^d=\frac{a^ds^{kd^2-jd}}{x_i^{kd-j}} \cdot \frac{s^{dj}}{x_i^j} \in \left (\frac{s^d}{x_i}\right)\mbox{,}$$
 and $a^ds^{kd^2-jd}/x_i^{kd-j}\notin I_{(x_i)} \mbox{,}$ because $\nu_{t'}(a^d/x_i^{kd-j})=0$.

If $g_1=bs^{dl}/x_i^l \notin I_{(x_i)}$,  then we obtain
$$
g_1^d\frac{a^ds^{kd^2-jd}}{x_i^{kd-j}}\notin I_{(x_i)}.
$$
But on the other hand,
$$ \frac{a^ds^{kd^2-jd}}{x_i^{kd-j}}g_1^d=g^d \frac{x_i^l}{s^{dj}}g_1^d=
\frac{c^ds^{(k+l-1)d^2}}{x_i^{(k+l-1)d}} \cdot \frac{s^{d^2-dj}}{x_i^{d-j}}\in I_{(x_i)} \mbox{,}$$
a contradiction. Thus, $g_1\in I_{(x_i)}$, and therefore $(s^d/x_i)$ is a $I_{(x_i)}$-primary ideal.

Now let $V$ denote an open subscheme in $X$, such that $P\in V$ ($P$ is a smooth point from definition \ref{geomdata}) and $V$ is normal (hence, Cohen-Macaulay). Then $X\backslash V$ is a closed subscheme with each irreducible component of dimension not greater than one. Let $E$ be an irreducible component of dimension one. Let $e$ denote the prime ideal of $E$ in the ring $A$ (the generic point of $E$ belongs to the affine set $X\backslash C=\Spec A$).  Take an element $a\in e$. Making an appropriate localization by a multiplicatively closed subset $S\subset A$, using \cite[prop. 4.9.]{AM}, we come to a ring $A_S$, where the primary decomposition of $(a)_S$ does not contain associated embedded ideals. So, all points on $\Spec A_S\cap E$ are Cohen-Macaulay. Therefore, there can be only finite number of not Cohen-Macaulay points on $X$. Since all points on $C$ are Cohen-Macaulay, we can find an open $U\supset C$ such that $U$ is a Cohen-Macaulay scheme.
\end{proof}

\begin{corol}
\label{freemod}
Let $\cf $ be a sheaf from geometric data  in definition~\ref{geomdata}. If $\cf$ is a coherent sheaf, then it is Cohen-Macaulay along the curve $C$. In particular, $\cf_P$ is a free $\co_P$-module.
\end{corol}

 \begin{proof} Since the sheaf $\cf$ is torsion free, then $\dim \cf_Q = \dim (\co_Q/ {\rm Ann} \, (\cf_Q)) =2$ for any point $Q\in C$. So, we have to show that $\dpth \cf_Q=2$.

 In the proof of theorem \ref{CM} we have shown that for any $Q\in C$ there is a regular sequence in $\co_{X,Q}$ coming from the sequence $s^d/x_i, bs^{dl}/x_i^l$ from $\tilde{A}_{(x_i)}$ for some $i$, where $bs^{dl}/x_i^l\notin I_{(x_i)}$ or, equivalently, $\nu_{t'}(b)=-dl$. Let's show that this sequence is regular also for $\cf_Q$.

As we have already remind in section~\ref{mapxi}, $\cf\simeq \Proj (\tilde{W})$. So, it is enough to prove that the element $bs^{dl}/x_i^l$ is not a zero divisor in the module $\tilde{W}_{(x_i)}/(s^d/x_i)\tilde{W}_{(x_i)}$. Let $g=as^{dk}/x_i^k\in \tilde{W}_{(x_i)}$ be an element such that $g\notin (s^d/x_i)\tilde{W}_{(x_i)}$. Note that this is equivalent to the condition $\nu_t(a)=-dkr+j$, where $0\le j<dr$ (see analogous arguments in the proof of theorem). But then $\nu_t(ab)=-d(k+l)r+j$, whence
by the same reason
$$
\frac{as^{dk}}{x_i^k} \cdot \frac{bs^{dl}}{x_i^l}\notin \left (\frac{s^d}{x_i}\right ) \tilde{W}_{(x_i)}.
$$
Thus,  $bs^{dl}/x_i^l$ is not a zero divisor and $\dpth \cf_Q=2$ for any $Q\in C$.

The last assertion follows from \cite[ch.~6, \S~16, exer.~4]{Mat}, because $P$ is a regular point.
\end{proof}

In view of theorem~\ref{techn5.2} it is important to compare the sheaf $\cl$ there (which is an analogue of the Baker-Akhieser module) with the sheaf $\cf$ appearing in the definition \ref{geomdata}. The following proposition gives a criterion answering the question when the sheaf $\cf$ from geometric data of rank $r$ is a coherent sheaf of rank $r$ on $X$.

\begin{prop}
\label{criterion}
Let $(X,C,P,\cf ,\pi , \phi )$ be a geometric data of rank $r$ from definition \ref{geomdata}.
The sheaf $\cf$ is coherent of rank $r$  on $X$ if and only if the self-intersection index $(C^2)=r$ on $X$.
\end{prop}

 \begin{proof} As we have already remind in the proof of theorem \ref{CM} (cf. \cite{Zhe2} and section~\ref{mapxi}), there are subspaces $A,W$ such that $X\simeq \Proj (\tilde{A})$, $\cf \simeq \Proj (\tilde{W})$. So, if the sheaf $\cf$ is coherent of rank $r$, then we have by \cite[ch.II, ex. 5.9]{Ha} $H^0(X, \cf (nC'))\simeq W_{nd}\simeq k[u,t]/(u,t)^{ndr+1}$ for $n\gg 0$. So, as in the proof of theorem \ref{techn5.2}, item 4 we obtain
$$
\chi (X, \cf (nC'))= \frac{(ndr+1) \cdot (ndr+2)}{2} \quad \mbox{for} \quad n \gg 0
$$
and from formula \eqref{rrf} it follows $C^2=r$.

Conversely, let the self-intersection index $(C^2)=r$ on $X$. Let $C'=dC$ be a very ample Cartier divisor. Then for any coherent sheaf $\cf'$ the coefficient at the degree $n^2$ of the polynomial $\chi (X, \cf'(nC'))$ is equal to $d^2(\rk \cf')r/2$ (see formula~\eqref{rrf}). Consider the sheaf $\cf' =\Proj (\tilde{W}')$, where $\tilde{W}'$ is a graded $\tilde{A}$-submodule in $\tilde{W}$ generated by elements from $W_{n}$ for sufficiently big $n$. Note that $\rk \cf' \ge r$. Indeed, there are elements $w_1,\ldots ,w_r$ in $W_{n}$ with $\nu_t(w_1)=-1, \ldots, \nu_t(w_r)=-r$ (because for $n=md$, $m\gg 0$, by definition~\ref{geomdata} and section~\ref{mapxi} we have $W_{md}\simeq H^0(X,\cf (mC'))\simeq k[[u,t]]/(u,t)^{mdr+1}$ and $W_{md}=f^{-md} H^0(X,\cf (mC'))$, where the space is considered as a subspace in $k[[u,t]]$ through the embedding from definition~\ref{geomdata}, item 6)
and therefore they are linearly independent over $\tilde{A}$. So, there is an embedding $\tilde{A}^{\oplus r}\hookrightarrow \tilde{W}'$ and since $\Proj$ is an exact functor (see \cite[prop.~2.5.4]{EGAII}), we obtain an embedding $\co_X^{\oplus r} \hookrightarrow \Proj (\tilde{W}')$, hence $\rk \cf' \ge r$. The same arguments show that the sheaf $\cf'|_C=\Proj (\gr \tilde{W}')$ on $C$ has rank greater or equal to $r$.

On the other hand, for big $n$ we have
$$\chi (X, \cf'(nC'))=\dim_k {W}'_{nd}\le \dim_k k[u,t]/(u,t)^{ndr+1} \mbox{,}$$
since ${W}'_{nd}\subset {W}_{nd}\simeq k[u,t]/(u,t)^{ndr+1}$ by section~\ref{mapxi}. So, the coefficient at the degree $n^2$ of the polynomial $\chi (X, \cf'(nC'))$ is less or equal than  $d^2r^2/2$. Hence, $\rk \cf'=r$ and $\rk \cf'|_C=r$. Then we also have $\chi (C, \cf'|_C(nC'))=r^2d^2n+c(\cf')$, where $c(\cf')\in \dz$.

Now consider two such coherent sheaves $\cf'_1\subset \cf'_2\subset \cf$. Then on $C$ we have the exact sequences
$$
0\rightarrow A(nC')\rightarrow \cf'_1|_C(nC')\rightarrow \cf'_2|_C(nC')\rightarrow B(nC')\rightarrow 0
$$
for all $n$, where $A$ and $B$ are coherent sheaves with finite support. Hence, we have
$$H^1(C, \cf'_1 |_C(nC'))=H^1(C, (\cf'_1/A) |_C(nC')) \mbox{,} $$ and for all $n\gg 0$ such that $H^1(C, (\cf'_1/A)|_C(nC'))=0$ we have $H^1(C, \cf'_2|_C(nC'))=0$. Let's fix such $n_0$. Note that this number depends only on $\cf'_1$, not on $\cf'_2$.

So, for all $n\ge n_0$ and for all coherent sheaves $\cf'_2\supset \cf'_1$ we have $H^1(C, \cf'_2|_C(nC'))=0$. Take some $q>n_0$ and consider the sheaf $\cf'_2 =\Proj (\tilde{W}')$, where $\tilde{W}'$ is a graded $\tilde{A}$-submodule in $\tilde{W}$ generated by elements from $W_{qd}$. Then we have
$$\chi (C, \cf'_2|_C(qC'))=\dim H^0(C, \cf'_2|_C(qC'))=qd^2r^2+c(\cf'_2)$$ for some $c(\cf'_2)\in \dz$.

Note that for all $n$ we have $\Proj (\tilde{W}'(nd))\simeq \Proj (\tilde{W'}^{(d)}(n))$ by~\cite[prop.~2.4.7]{EGAII} (recall that $\tilde{W}'(nd)$ is equivalent to $\oplus_{i=0}^{\infty}W'_{i+nd}s^i$), and $\Proj (\tilde{W'}^{(d)}(n))\simeq \Proj (\tilde{W'}^{(d)})(n)\simeq \cf'_2 (nC')$ by~\cite[ch.~II, prop.~5.12]{Ha}. So,
$$H^0(X, \cf'_2 (nC'))= H^0(X, \Proj (\tilde{W}'(nd))) \mbox{.}$$

\begin{lemma}{(cf. \cite[lemma 3.5]{Zhe2})}
We have
$$H^0(X, \Proj (\tilde{W}'(nd)))= {W}'_{nd}.$$
\end{lemma}

\begin{proof} By definition, we have ${W}'_{nd}=\tilde{W'}(nd)_0\subset H^0(X,\Proj (\tilde{W}'(nd)))$.

Let $a\in H^0(X, \Proj (\tilde{W}'(nd)))$, $a\notin {W}'_{nd}$. Then $a=(a_1,\ldots ,a_k)$, where $a_i\in (\tilde{W}'(nd))_{(x_i)}$,  and $x_i\in \tilde{A}_d$ are generators of the space $\tilde{A}_d$ such that $x_1=s^d$, $x_i=x'_is^d$ (where $x'_i\in A_d$) and $a_i=a_j$ in $\tilde{A}_{x_ix_j}$.

We have $a_i=\tilde{a}_i/x_i^{k_i}$ (where $\tilde{a}_i=a'_is^{k_id}$, $a'_i\in W'_{k_id+nd}$), $a_1=\tilde{a}_1/s^{k_1}$ and $k_1>0$ since $a\notin {W}'_{nd}$. Indeed, if $\tilde{a}_1\in \tilde{W}'(nd)_{0}={W}'_{nd}$, then $a=\tilde{a}_1$ since $\tilde{W}'$ is a torsion free $\tilde{A}$-module, a contradiction.
So, we have
$${a}'_1\in {W}'_{k_1+nd}\backslash {W}'_{k_1+nd-1}.
$$

Then for $x'_i\in {A}_d\backslash {A}_{d-1}$ (such a generator $x_i$ exists because all elements from $A_{d-1}\subset A_d$ lie in the ideal that defines the divisor $C$) we have ${x'_i}^{k_i}\in {A}_{dk_i}\backslash {A}_{dk_i-1}$ and therefore ${a}'_1 {x'_i}^{k_i}\in {W'}_{k_1+dk_i+nd}\backslash {W'}_{k_1+dk_i+nd-1}$. On the other hand, we have the equality $\tilde{a}_1 x_i^{k_i}= \tilde{a}_i s^{k_1}$, hence $a'_1{x'_i}^{k_i}=a'_i$, but
$${a}'_i \in {W'}_{dk_i+nd}\subset {W}'_{k_1+dk_i+nd-1},$$
a contradiction. So, $a\in {W}'_{nd}$.
\end{proof}

Now we have that
$$H^0(C, \cf'_2|_C(qC'))\supset H^0(X, \cf'_2(qC'))/H^0(X, \cf'_2((q-1)C')).$$
By lemma and by definition of the sheaf $\cf'_2$ we have
$$
H^0(X, \cf'_2(qC'))/H^0(X, \cf'_2((q-1)C'))={W}'_{qd}/{W}'_{(q-1)d}={W}_{qd}/{W}_{(q-1)d}.$$
So, we obtain $c(\cf'_2)\ge \dim ({W}_{qd}/{W}_{(q-1)d})-qr^2d^2$.

On the other hand, for big $n$ we have
$$
H^0(C, \cf'_2|_C(nC'))= H^0(X, \cf'_2(nC'))/H^0(X, \cf'_2((n-1)C'))={W}'_{nd}/{W}'_{(n-1)d}\subset {W}_{nd}/{W}_{(n-1)d}
$$
and $\dim_k ({W}_{nd}/{W}_{(n-1)d})-nr^2d^2=\mbox{const}=l$ for all $n\ge 0$. Therefore
$$c(\cf'_2)=\dim H^0(C, \cf'_2|_C(nC'))-nr^2d^2\le l \mbox{.}$$
 Hence $c(\cf'_2)=l$, ${W}'_{nd}/{W}'_{(n-1)d}= {W}_{nd}/{W}_{(n-1)d}$ for all $n\ge 0$, and consequently $\tilde{W}'=\tilde{W}$, i.e. $\cf =\cf'_2$ is a coherent sheaf of rank $r$ on $X$.
\end{proof}

\begin{nt}
The sheaf $\cf$ from geometric data of rank $r$ may be not coherent, as the following example shows. Let $W=\langle u^it^{-j}\mbox{\quad} |i,j \ge 0, i-j\le 0\rangle $ and $A=k[ut^{-2},t^{-2}]$ be two subspaces in $k[[u]]((t))$. Then it is easy to see that $W$ is not a finitely generated $A$-module. So, the geometric data constructed by these subspaces (see \cite[th.3.3]{Zhe2}) will contain a quasicoherent, but not coherent sheaf $\cf$.

The coherence of the sheaf $\cl$ constructed by a ring of partial differential operators in theorem \ref{techn5.2} followed from special conditions on the ring (see item 1 of this theorem). These conditions may be not true for a general $1$-quasi elliptic strongly admissible ring (see theorem \ref{dannye2}) even if such a ring is a ring of partial differential operators, as the example above shows: indeed, the ring $A$ above corresponds to the ring $B=k[\partial_2^2, \partial_1\partial_2]$, see \cite[th.3.2]{Zhe2}. Nevertheless, the proposition above guarantees that for a surface and a divisor satisfying certain geometric conditions the sheaf $\cf$ must be coherent.

Now a natural question arises: how are connected the variety $X$, the divisor $C$ and the sheaf $\cl$ from theorem \ref{techn5.2} (if $\dim X=2$), constructed by a ring of partial differential operators from this theorem, and corresponding objects of the geometric data constructed by the same ring in theorem \ref{dannye2}? The proposition below gives the answer.
\end{nt}

\begin{prop}
\label{connection}
Let $B$ be a commutative subring of partial differential operators in two variables which satisfies the conditions from theorem \ref{techn5.2} and from theorem \ref{dannye2} (cf. the beginning of this subsection).
Then the triple $(X,C,\cl )$ from theorem \ref{techn5.2} is isomorphic to the triple $(X,C,\cf )$ (a part of geometric data) from theorem \ref{dannye2}.
\end{prop}
\begin{proof} Recall that the surface and divisor from theorem \ref{dannye2} are constructed by the graded ring $\tilde{A}$ which is defined by the ring $A\subset k[[u]]((t))$ (cf. section~\ref{mapxi}, see the end of the section \ref{rings}).
 So, there is a natural isomorphism $\alpha :B\rightarrow A$, $b\mapsto \psi_1(SbS^{-1})$ which induces an isomorphism $\tilde{B}\simeq \tilde{A}$, whence we have the isomorphism of surfaces and divisors.

Moreover, note that the map $\varphi :L\rightarrow W$, $l\mapsto \psi_1(lS^{-1} \mod x_1D +x_2D)$ gives an isomorphism between the $B$-module $L$ and $A$-module $W$, because $SBS^{-1}$ is a subring of pseudo-differential operators with constant coefficients, see the proof of theorem~3.2 in \cite{Zhe2}, and
\begin{multline*}
\varphi (lb)= \psi_1(lb S^{-1} \mod x_1D +x_2D) = \psi_1(lS^{-1}(SbS^{-1})\mod x_1D +x_2D) =\\ =\psi_1 (lS^{-1}\mod x_1D +x_2D) \psi_1(SbS^{-1})=\varphi (l)\alpha (b) \mbox{.}
\end{multline*}
So, $\varphi$ induces an isomorphism of sheaves $\cl$ and $\cf$, and this isomorphism is compatible with the isomorphism of surfaces.
\end{proof}

\begin{nt}
\label{primer}
The rank of the sheaf $\cf$ from geometric data in definition \ref{geomdata} may be greater than the rank of the data even if the sheaf $\cf$ is coherent (as it easily follows from the arguments of proposition \ref{criterion}, the rank of $\cf$ can not be less than the rank of the data).

For example, consider the ring of PDO $B=k[\partial_2, \partial_1\partial_2+\partial_1^2]$. It satisfies the conditions of theorem \ref{techn5.2}. This is also a strongly admissible ring in the sense of \cite[def. 2.11]{Zhe2} and $N_B=1$, thus the rank of the corresponding geometric data is one (see \cite[th.3.3]{Zhe2}). On the other hand, for big $m$ we have $\dim_k B_m\sim m^2/4$ and $\dim_kL_m\sim m^2/2$ (in the notation of theorem \ref{techn5.2}). Therefore, $\rk \cf =2$ ( see the proof of theorem \ref{techn5.2}).

So, in general we have $\rk \cf \ge r$, where $r$ is the rank of the data.
\end{nt}

As it follows from propositions above, theorem \ref{techn5.2} and theorem \ref{dannye2}, to find new explicit
examples of partial differential operators in two variables, it is important to solve the following algebro-geometric
problem (especially in view of finding new quantum algebraically completely integrable systems).

\begin{pb}
\label{problem1}
It would be nice to classify (and/or to find a way how to construct) projective irreducible surfaces $X$ that have an irreducible ample effective $\dq$-Cartier divisor $C$ not contained in the singular locus of $X$, with the self-intersection index $(C^2)=1$ on $X $, and such that there are coherent sheaves satisfying the conditions of item~6 of definition \ref{geomdata}. Since any unirational curve is a rational curve (by L\"uroth's theorem), the curve $C$ has to be a rational curve (see item~\ref{uni} of theorem~\ref{techn5.2})
\end{pb}

\subsection{Geometric ribbons}
\label{ribbons}

In this section we give a construction of a ribbon and a torsion free sheaf on it (see \cite{Ku}) determined by the geometric data defined in \ref{geomdata}. This construction generalizes the constructions of ribbons coming from geometric data given in \cite{Ku1}.

We recall the general definition of ribbon from~\cite{Ku}.

\begin{defin}[\cite{Ku}] \label{ribbon}
A ribbon $(C, \ca )$ over a field $k$ is given by the following data.
\begin{enumerate}
\item\label{a}
 A reduced algebraic curve $C$ over $k$.
\item \label{b} A sheaf $\ca$ of commutative $k$-algebras
on $C$.
\item \label{c} A descending sheaf filtration $(\ca_i)_{i\in \sdz}$ of $\ca$ by
$k$-submodules which satisfies  the following axioms:
\begin{enumerate}
\item  \label{i} $\ca_i\ca_j \subset \ca_{i+j}$, $1\in \ca_0$ (thus $\ca_0$
is a subring, and for any $i \in \dz$ the sheaf $\ca_i$ is a
$\ca_0$-submodule);
\item \label{ii}
$\ca_0/\ca_1$ is the structure sheaf $\co_C$ of $C$;
\item
\label{iii} for each $i$ the sheaf $\ca_i/\ca_{i+1}$ (which is a
$\ca_0/\ca_1$-module by (\ref{i})) is a coherent sheaf on $C$,  and the sheaf $\ca_i/\ca_{i+1}
$ has no coherent subsheaf with finite support, and is
isomorphic to $\co_{C}$ on a dense open set;
\item
\label{iv}  $\ca =\limind\limits_{i \in \sdz} \ca_i$, and
$\ca_i=\limproj\limits_{j>0}\ca_i/\ca_{i+j}$ for each $i$.
\end{enumerate}
\end{enumerate}
\end{defin}

We note that if $C$ is also an irreducible curve, then condition~\eqref{iii} from definition~\ref{ribbon} is equivalent to the condition that for each $i$ the sheaf
$\ca_i/\ca_{i+1}$ is a rank $1$ torsion free coherent sheaf on $C$.

We recall also the general definition of a torsion free sheaf on a ribbon from~\cite{Ku}.

\begin{defin} \label{tfsh}
Let $\xo_{\infty}=(C,\ca )$ be a ribbon over a field $k$. We say
that $\cnn$ is a torsion free sheaf of rank $r$ on $\xo_{\infty}$ if
  $\cnn$ is a sheaf of $\ca$-modules on $C$ with a descending
filtration $(\cnn_i)_{i\in \sdz}$ of $\cnn$ by $\ca_0$-submodules
which satisfies the following axioms.
\begin{enumerate}
\item
 $\cnn_i\ca_j\subseteq \cnn_{i+j}$ for any $i,j$.
\item \label{ittf}
For each $i$ the sheaf $\cnn_i/\cnn_{i+1}$ is a coherent sheaf on $C$,
 and   the sheaf
$\cnn_i/\cnn_{i+1}$ has no coherent subsheaf with finite
support, and is isomorphic to $\co_{C^{\oplus r}}$ on a dense open
set.
\item
 $\cnn =\limind\limits_i \cnn_i$, and
$\cnn_i=\limproj\limits_{j>0}\cnn_i/\cnn_{i+j}$ for each $i$.
\end{enumerate}
\end{defin}

We note that if $C$ is also an irreducible curve, then condition~\eqref{ittf} from definition~\ref{tfsh} is equivalent to the condition that for each $i$ the sheaf
$\cnn_i/\cnn_{i+1}$ is a rank $r$ torsion free coherent sheaf on $C$.

Let $(X,C,P,\cf ,\pi ,\phi )$ be a geometric data from definition \ref{geomdata}, and $\cf$ is a coherent sheaf of rank $r$ on $X$ ($r$ is not necessary equal to the rank of geometric  data).

Then we can define a ribbon $ \xo_{\infty}=(C,\ca )$ and a torsion free sheaf $\cn$ on it as follows. Recall that subspaces $W,A$ in $k[[u]]((t))$ are defined by this data (see section~\ref{mapxi}) and $X\simeq \Proj (\tilde{A})$, $\cf \simeq \Proj (\tilde{W})$.

  Let $X_{\infty}=(C,\hat{\ca}_0)$ be a formal scheme, the formal completion of $X$  along $C$. Let's define a family of  sheaves $\ca_i=\hat{\cb}_{-i}$ on $X_{\infty}$ (see the end of section \ref{mapxi}).
  Since the functor $\cb_i\mapsto \hat{\cb}_i$ is an exact functor (see e.g. \cite[Corol.~9.8]{Ha}), we have ${\ca}_i\supset {\ca}_{i+1}$ for all $i$. Obviously, for all $i$ the sheaves $\cb_{i}/\cb_{i-1}\simeq \hat{\cb}_{i}/\hat{\cb}_{i-1}$ are torsion free on $C$ and $\ca_{0}/\ca_{1}\simeq \co_C$. The multiplication in the ring $A$ induces a multiplication map $\cb_i\cb_j \subset \cb_{i+j}$ and hence a multiplication map ${\ca}_i{\ca}_j\subset {\ca}_{i+j}$. So, the sheaf $\ca =\limind  {\ca}_i$ defines a structure of a ribbon $(C,\ca )$ according to the definition \ref{ribbon}. Analogously, let's define the sheaf $\cn =\limind {\cn}_i$, where $\cn_i=\hat{\cf}_{-i}$. The sheaves ${\cn}_i/{\cn}_{i+1}$ are torsion free coherent sheaves for all $i$ (cf. the end of section \ref{mapxi}).

Let's show that $\cn$ is a torsion free sheaf of rank $r$ on the ribbon $(C,\ca )$ in the sense of \cite[def.~11]{Ku} and that
the point $P$ is smooth for the sheaf $\cn$ in the sense of \cite[def.~12]{Ku}.
Since $P$ is a smooth point on $C$, we have to check that the map
$$
(\widehat{\cf_i/\cf_{i-1}})_P\otimes_{\hat{\co}_{C,P}}(\widehat{\cb_j/\cb_{j-1}})_P\longrightarrow (\widehat{\cf_{i+j}/\cf_{i+j-1}})_P
$$
induced by the multiplication map $\cf_i\cdot \cb_j\subset \cf_{i+j}$ is an isomorphism. This follows from the facts that $(\cf_i/\cf_{i-1})_P\simeq \co_{C,P}^r$, $(\cb_j/\cb_{j-1})_P\simeq \co_{C,P}$ (since they are torsion free and $P$ is a smooth point on $C$). From the last facts it follows also that $\cn$ is a torsion free sheaf of rank $r$ (cf. corollary \ref{freemod}).

So, we obtain the following proposition.
\begin{prop}
\label{construction}
Given a geometric data $q=(X,C,P,\cf ,\pi , \phi )$ from definition \ref{geomdata}, where $\cf$ is a coherent sheaf of rank $r$, we can canonically define a geometric data  $\tilde{q}$ consisting of a ribbon $(C,\ca )$ over a field $k$, a torsion free sheaf $\cn$ of rank $r$, a smooth $k$-point $P$ of the sheaf $\cn$, formal local parameters $u',t'$ and a trivialization $e_P: \hat{\cn}_{0,P} \rightarrow \hat{\ca}_{0,P}^r\simeq  k[[u',t']]^r$
(cf. \cite[def.14]{Ku}). So, we have a map
$$
\Phi : q \mapsto \tilde{q}.
$$

The construction of the data $\tilde{q}$ generalizes the construction of a geometric ribbon given in \cite[ex. 1]{Ku}. The data $\tilde{q}$ satisfies conditions of \cite[th.1]{Ku}.
\end{prop}

Note that if we start with a data $q'=(CM(X),C,P, CM(\cf) ,\pi ,\phi )$,
where $CM(X)$ and $CM(\cf)$ are the Cohen-Macaulaysations of $X$ and $\cf$ (see
appendix~B), then the data $\tilde{q'}=\tilde{q}$ (because it
follows from the construction of the data, theorem \ref{CM} and from
results of appendix~B). Remarkably, the following is true:
\begin{prop}
\label{reconstruction}
If $q,q'\in \cq_r$  and surfaces $X,X'$ are Cohen-Macaulay, and the data $(C,\ca ,P, u,t)$, $(C',\ca' ,P', u',t')$ constructed by the map $\Phi$ from $q$ and $q'$ are isomorphic (cf. \cite[def.14]{Ku}), then $X$ is isomorphic to $X'$.
\end{prop}
\begin{proof} The idea of the proof is to apply arguments from \cite[th.5,6]{Os} to the data $(X, dC, \tilde{P}, \co_X)$, $(X', d'C', \tilde{P}', \co_{X'})$, where $dC$, $d'C'$ are ample Cartier divisors and $\tilde{P}$, $\tilde{P}'$ are ample Cartier divisors on $dC$,$d'C'$ induced by Cartier divisors $P$,$P'$ on $C$,$C'$ and local parameters $u$, $u'$ (cf.~\cite[lemma 5]{Ku}). Since the ribbon data are isomorphic, their images under the generalized Krichever map coincide (see \cite[th.1]{Ku}). In this situation the algebras $A_{(0)}(\co_X)$, $A_{(0)}(\co_{X'})$ coincide (see the proof of \cite[th.6]{Os}), therefore the surfaces  $X,X'$ defined by these algebras will be isomorphic.

\end{proof}

\begin{nt}
As it follows from \cite[th.5,6]{Os}  $A_{(0)}(\co_X)=\mathbb{A}\cap k[[u]]((t))$ (in the notation of introduction). Thus, we obtain the equation \eqref{intersection}.
\end{nt}

\subsection{Glueing construction}
\label{glueing}

As in the one-dimensional case one gets interesting solutions of the KP-equation by algebraic curves which are obtained from glueing points (to get cuspidal or nodal curves, with non-trivial Jacobians and compactified Jacobians, see \cite{Ma,SW}), we hope that a similar construction for surfaces, starting from $\dpp^2$ or rational or other surfaces yields examples in our case to solve problem~\ref{problem1}. This hope is confirmed by the fact that practically all known examples of commutative rings of PDOs have such glued surfaces as spectral manifolds, see examples below.

We need a construction where we glue curves on a surface, or glue points of a curve on a surface. The problem is: given a projective surface $\tilde{X}$, a one-dimensional closed subscheme $Y$ and a surjective finite morphism $p:Y\rightarrow C$ to a curve $C$, find a surface $X$ with a curve $C\subset X$ and a morphism $n:\tilde{X}\rightarrow X$ such that $n(Y)=C$ (and $n|_Y$ is a given morphism $p$) and $n:\tilde{X}\backslash Y \simeq X\backslash C$ is an isomorphism.

In the work \cite{Fe} is given a construction of glueing closed subschemes on a given scheme. It can be done for many schemes with mild conditions, and the construction is a natural generalization of the construction for curves given in \cite{Se}:

\begin{theo}{(\cite[th.5.4]{Fe})}
\label{Ferrand}
Let $X'$ be a scheme, $Y'$ be a closed subscheme in $X'$ and $g:Y'\rightarrow Y$ be a finite morphism. Consider
the ringed space $X=X'\sqcup_{Y'}Y$ (the amalgam) and the cocartesian square
$$
\begin{diagram}
\node{Y'}
\arrow{s} \arrow{e,t}{g} \node{Y} \arrow{s,r}{u} \\
\node{X'} \arrow{e,t}{f} \node{X}
\end{diagram}
$$
Suppose that the schemes $X'$ and $Y$ satisfy the following property:
$$
\leqno{(AF)} \mbox{\quad All finite sets of points are contained in an open affine set}.
$$
Then:

a) $X$ is a scheme satisfying $(AF)$;

b) the square above is cartesian;

c) the morphism $f$ is finite and $u$ is a closed embedding;

d) $f$ induces an isomorphism of $X' - Y'$ on $X-Y$.
\end{theo}

Often the schemes after the glueing procedure are proper but not projective (see  \cite[\S~6]{Fe} and also \cite[prop.~5.6]{Fe}). Here we rewrite the construction from \cite{Fe} for surfaces with some extra conditions in another, but equivalent way.

We will need some more assumptions. The idea is to construct $X$ first as a topological space by taking the quotient space with respect to the equivalence relation $\triangle_X\cup (id\times p)^{-1}(\Gamma_p)$ in $X\times X$ (here $\Gamma_p\subset Y\times C$ is a graph of $p$). Then we have a topological quotient map $\tilde{X}\rightarrow X$, but it is not obvious how to make $X$ to a scheme.

For this we make the assumption that $p$ extends to a morphism $\tilde{p}:\tilde{X}_0\rightarrow C$, where $\tilde{X}_0\subset \tilde{X}$ is a (Zariski) open neighbourhood of $Y$.

Without loss of generality we may assume that $\tilde{X}_0=\tilde{X}$: by replacing $\tilde{X}$ by the closure of the graph of $\tilde{p}$ in $\tilde{X}\times C$ we can modify $\tilde{X}$ to $\tilde{\tilde{X}}\rightarrow \tilde{X}$ such that $\tilde{p}\sigma$ extends to a morphism of $\tilde{\tilde{X}}$ to $C$, this modification is outside of $\tilde{X}_0$ and after the glueing construction one can reverse this modification.

Then, making this topological construction as above, we get a factorization of the underlying continue map of $\tilde{p}$ as
$$
\begin{diagram}
\node{\tilde{X}}  \arrow{e,t}{\tilde{p}} \arrow{s,l}{n}  \node{C} \\
\node{X}   \arrow{ne}
\end{diagram}
$$
and one should define $\co_X=p^{-1}\co_C+n_*I_Y$ ($I_Y\subset \co_{\tilde{X}}$ ideal sheaf of $Y$).

Description with affine covering: for each $c\in C$ there exists an affine neighbourhood $V\subset C$ and  an affine neighbourhood $\tilde{U}$ of $p^{-1}(c)$ in $\tilde{X}$ such that $\tilde{U}\cap Y=\tilde{p}^{-1}(V)\cap Y$ (first: take any affine neighbourhood $U$ of $p^{-1}(c)$, then $Y\backslash U$ is closed and $F=p(Y\backslash U)$ is closed in $C$ and disjoint to $c$. Then we take $V\subset C\backslash F$ an affine neighbourhood of $c$, and $U=\tilde{U}\cap \tilde{p}^{-1}(V)$).
Then the diagram  of maps
$$
\begin{diagram}
\node{\tilde{U}}  \arrow{e}  \node{V} \\
\node{\tilde{U}\cap Y}  \arrow{n,J} \arrow{ne}
\end{diagram}
$$
corresponds to homomorphisms of coordinate rings
$$
\begin{diagram}
\node{\tilde{A}}  \arrow{s}  \node{R} \arrow{w}  \arrow{sw} \\
\node{\tilde{A}/I}
\end{diagram}
$$
and we define $A=R+I\subset \tilde{A}$. We have to prove that $A$ is a finitely generated $R$-algebra (and $\Spec (A)\subset X$ as a topological space).
\begin{lemma}
\label{lemma1}
Let $R\subset \tilde{A}$ be finitely generated domains over a field $k$ of Krull dimension 1 and  2, $I$ an ideal in $\tilde{A}$ such that $R\subset \tilde{A}/I$ is finite. Then $A=R+I\subset \tilde{A}$ is finitely generated and $\tilde{A}$ is finite over $A$.
\end{lemma}
\begin{proof} We can find $f_1, f_2\in \tilde{A}$ such that \\
1) $\tilde{A}$ is finite over $E=k[f_1,f_2]$;\\
2) $f_1\in I$ (see \cite[ch.V, \S 3, th.1]{Bu}).

Since $\tilde{A}/I$ is finite over $R$, there exist $a_1,\ldots ,a_k\in {A}$ such that $f_2^k+a_1f_2^{k-1}+\ldots +a_k=b\in I$; hence if $A_0=k[f_1,a_1,\ldots ,a_k,b]$ then $A_0\subset R+I$ and $\tilde{A}$ is finite over $A_0$ ($A_0\subset A_0[f_2]\subset \tilde{A}$), therefore $A=R+I$ is finitely generated over $k$.
\end{proof}

So, $\tilde{A}$ is a (partial) normalization and $I$ is the conductor from $A$ to $\tilde{A}$. One checks easily that \\
1) $\tilde{X}\stackrel{n}{\rightarrow} X$ is universally closed;\\
2) $X$ is separated:
$$
\begin{diagram}
\node{\tilde{X}\times \tilde{X}} \arrow{e,t}{(1)} \node{\tilde{X}\times X} \arrow{e,t}{(2)} \node{X\times X} \\
\node{\tilde{\bigtriangleup}}  \arrow{n,l,J}{\alpha} \arrow{e}  \node{\Gamma_n } \arrow{n,l,J}{\beta} \arrow{e} \node{\bigtriangleup}
\arrow{n,r,J}{\gamma}
\end{diagram}
$$
Since $(1)$ and $(2)$ are closed maps and $\Gamma_n$, $\bigtriangleup$ are locally closed, then the maps $\beta ,\gamma$ are closed embeddings.

So, the map $n$ is proper, and $X$ is a separated scheme of finite type over $k$. Since $n$ is also finite surjective and $\tilde{X}$ is proper over $k$, it follows that the structure map $X\rightarrow \Spec k$ is universally closed, hence $X$ is a proper over $k$ scheme.

Having this construction and general theorem \ref{Ferrand} we can give examples:
\begin{ex}
\label{1example}
Consider $\tilde{X}=\dpp^2$ with a morphism $p:\tilde{X}-(0:0:1)\rightarrow C=\dpp^1$, $p(x:y:z)=(x:y)$ and two lines $Y=2\dpp^1$, where $\dpp^1=(x:y:0)$ in $\dpp^2$. Clearly, $Y$ is an ample Cartier divisor on $\tilde{X}$.

In this case a Cartier divisor $\tilde{C}$ on the glued surface $X$ is defined which is given by the same equations in the induced local covering  (so, $Y=n^*\tilde{C}$). Since $Y$ is an ample divisor and $n$ is a finite surjective morphism of proper schemes, the divisor $\tilde{C}$ is also ample (see~\cite[ex.~5.7., ch.~3]{Ha}). Therefore, $X$ is projective, and the cycle map from appendix~A  gives $Z(\tilde{C})=2C$. So, $C$ is an ample $\dq$-Cartier divisor on $X$.
\end{ex}

\begin{ex}
Let $X'$ be a projective surface over $k$. Let $Y' = \coprod\limits_{i=1}^k \Spec I_i $, where each $I_i$ is a local Artinian $k$-algebra, be a zero-dimensional subscheme of $X'$. (For example, $Y'$ can be a finite number of closed points on $X'$, or $Y'$ can be a zero-dimensional subscheme concentrated in a closed point on $X'$.)  Let $Y=\Spec k$ be a point and $g:Y'\rightarrow Y$ a finite morphism.
Then the scheme $X$ constructed as in theorem \ref{Ferrand} is proper over $k$ (see e.g. \cite[\S~6.1]{Fe}) and it is projective by the same arguments as in the previous example (i.e. by~\cite[ex.~5.7., ch.~3]{Ha}).
\end{ex}

\begin{ex}
More generally, let $\pi :X'\rightarrow X$ be a normalisation of algebraic variety $X$. Then the conductor $\ci \subset \co_X$, $\ci \subset \pi_*\co_{X'}$  defines closed subschemes $Y\subset X$ and $Y'\subset X'$ with the finite morphism $Y'\rightarrow Y$. Then (keeping in mind theorem \ref{Ferrand}) we have $X\simeq  X'\sqcup_{Y'}Y$, because for any scheme $Z$ we obviously have
$$
\Hom (X,Z) \simeq \Hom (X',Z)\times_{\Hom (Y',Z)}\Hom (Y,Z) \mbox{\quad via ($\psi :X\rightarrow Z$) $\mapsto$ $(\psi \pi , \psi |_Y)$}.
$$

This remark can be applied to a series of known examples of
commuting PDOs containing the Schr\"odinger operator with a special
potential. In these examples the commutative rings are known to be
isomorphic to the rings of quasi-invariants (see e.g. \cite{FV} or
\cite{Ch} for details) all of which have the normalization equals to
$\dc [x_1,\ldots ,x_n]$ (moreover, they are Cohen-Macaulay as it is
shown e.g in \cite[th.8]{FV}). So, in all these examples the (affine)
spectral varieties are the glued $\da^n$. Note that for $n>1$ one
has to glue subschemes of codimension one in a normal variety to
obtain a Cohen-Macaulay variety which is not normal variety. Indeed,
if we glue subschemes of codimension more than one and obtain a
Cohen-Macaulay variety, it should be normal by Serre's criterion, a
contradiction.
\end{ex}

\begin{ex}
Using the remark from previous example let's note that the example 4.2 from \cite{Zhe2} can be obtained by glueing two lines $2\da^1$ on $\da^2$. Recall that in this example the (modified) Schur pair $(A,W)$ was constructed, where
$$
W=\langle 1+t, t^{-i}u^j, i\ge 1, 0\le j\le i \rangle \subset k[[u]]((t)), \mbox{\quad} A=k[t^{-2},t^{-3},ut^{-2}]\simeq k[x^2,x^3,h]
$$
The conductor of $A$ in the normalisation $\tilde{A}\simeq k[x,h]$ is $(x^2)$ (in $\tilde{A}$). So, $Y'=2\da^1$ and $Y=\da^1$ in the notations of the previous example. Recall also that the corresponding ring of commuting operators (though not PDO) is $B=k[P,Q,P']$, where
$$
P=\partial_2^2-2\frac{1}{(1-x_2)^2}(:\exp (-x_1\partial_1):),
$$
$$
Q=\partial_1\partial_2+\frac{1}{1-x_2}(:\exp (-x_1\partial_1):)\partial_1,
$$
$$
P'=\partial_2^3-3\frac{1}{(1-x_2)^{2}}(:\exp (-x_1\partial_1):)\partial_2-3\frac{1}{(1-x_2)^{3}}(:\exp (-x_1\partial_1):),
$$
and $(:\exp (-x_1\partial_1):)=1-x_1\partial_1+x_1^2\partial_1^2/2!-x_1^3\partial_1^3/3!+\ldots$.
We note that the space $W$ is infinitely generated over $A$; hence it defines a quasi-coherent sheaf on the projective surface. At last, it's easy to see that $A$ is Cohen-Macaulay (as a polynomial ring over $k[x^2,x^3]$, which is also Cohen-Macaulay, see \cite[th.33]{Mat}).
\end{ex}

We hope to return to the problem of constructing new examples in future works.

\section*{Appendix A. Cycle map}
%\label{cyclemap}

For reader's convenience, we'll give below some facts about the cycle map for singular projective varieties (cf. \cite[\S 2.1]{Fu}).

Let $X$ be a projective irreducible $n$-dimensional variety over a field $k$. Let $\Div (X)$ be a group of Cartier divisors (equal to $H^0(X,\ k(X)^*/\co_X^*)$). Thus, a divisor $D$ is given by data $(U_{\alpha},f_{\alpha})$, where $f_{\alpha}\in k(X)^*$, $\{U_{\alpha}\}$ is an open covering of $X$, and $f_{\alpha}/f_{\beta}\in \co_X^*(U_{\alpha \beta})$.

If $C\subset X$ is an irreducible closed subvariety of codimension $1$,  define the order
$$
\ord_C(D)= l_{\co_{X,C}}(\frac{1}{f_{\alpha}}\co_{X,C}/\co_{X,C}\cap \frac{1}{f_{\alpha}}\co_{X,C})-l_{\co_{X,C}}(\co_{X,C}/\co_{X,C}\cap \frac{1}{f_{\alpha}}\co_{X,C}),
$$
where $l_{\co_{X,C}}$ denotes the length of a $\co_{X,C}$-module, and $\alpha$ is chosen so that $C$ meet $U_{\alpha}$ (the choice is irrelevant for the definition of $\ord_C(D)$).

Set the cycle corresponding to a divisor $D\in \Div(X)$ as
$$
{\rm Z}(D)=\sum_{\mbox{\small codim} \: C=1}\ord_C(D) \, C \mbox{.}
$$

So, we have a homomorphism ${\rm Z} \ : \ \Div(X) \rightarrow Z^{1}(X)$ to the group of cycles of codimension one on $X$ (which is the free abelian group generated by all irreducible subvarieties of codimension $1$ on $X$). The group $Z^{1}(X)$ is also called the group of Weil divisors on X and is denoted also by
 ${\rm WDiv}(X)$. Let's compute the kernel of the homomorphism $\rm Z$.

If $R$ is a one-dimensional local domain and $g\in \Quot (R)$, then we have to compute
$$
L_g= l_R(R/R\cap gR)-l_R(gR/R\cap gR) \mbox{,}
$$
which can be also expressed as follows. If we choose an element $a\in R$ such that $ag=b\in R$, then
\begin{multline*}
L_g=l_R(aR/aR\cap bR)-l_R(bR/aR\cap bR)= \\ =
l_R(R/aR\cap bR)-l_R(R/aR)-(l_R(R/aR\cap bR)-l_R(R/bR)) = l_R(R/bR)-l_R(R/aR) \mbox{.}
\end{multline*}
One checks that if $R\subset \tilde{R}\subset \Quot (R)$, and if  $\tilde{R}$ is a finitely generated  $R$-module, then
$$
L_g= l_R(R/bR)-l_R(R/aR)=l_R(\tilde{R}/b\tilde{R})-l_R(\tilde{R}/a\tilde{R})  \mbox{.}
$$
If $\tilde{R}$ is the integral closure of $R$, then  one sees that  $L_g=0$ if  $b/a=g\in \tilde{R}^*$. Thus
$$
\Ker (\Div (X)\stackrel{\rm Z}{\longrightarrow} Z^1(X))\supset H^0(X, \pi_*\co_{\tilde{X}}^*/\co_X^*)\simeq \Ker (\Pic(X)\stackrel{\pi^*}{\longrightarrow}\Pic (\tilde{X})),
$$
where $\tilde{X}\stackrel{\pi}{\rightarrow}X$ is the normalization of $X$. (We used an exact sequence of sheaves on $X$:
$$
0 \lto \co_X^* \lto \pi_* \co_{\tilde{X}}^* \lto \pi_* \co_{\tilde{X}}^* / \co_X^* \lto 0 \mbox{,}
$$
and the fact from~\cite[Th.~38]{Mat}: if $A$ is an integrally closed Noetherian domain, then   \linebreak $A = \bigcap\limits_{\mbox{\small ht} \, \nu =1} A_{\nu}$, where the intersection is taken over all prime ideals of height one.)

One can define the semigroup of effective Cartier divisors $\Div^+(X) \subset \Div(X)$ (given by data $(U_{\alpha}, f_{\alpha})$, where
$f_{\alpha} \in \co_X(U_{\alpha}) \cap k(X)^*$), and the semigroup of effective Weil divisors ${\rm WDiv^+(X)} \subset {\rm WDiv(X)}$  (given by the formal finite sums of cycles of codimension one with positive integer coefficients). It is easy to see that ${\rm Z} ({\Div^+(X)}  ) \subset {\rm WDiv^+(X)}$. Moreover,
from the above description of the map ${\rm Z}$ it follows that the cycle map $\rm Z$ restricted to the semigroup of effective Cartier divisors $\Div^+(X)$ is an injective map to the semigroup of effective Weil divisors $\WDiv^+(X)$ not contained in the singular locus.

\section*{Appendix B. Cohen-Macaulaysation}
%\label{Cohen}

In this appendix we give a construction of Cohen-Macaulaysation of a surface and of a sheaf. This construction can be probably considered as one of "folklore results", confer for example discussions about Macaulayfication functor in \cite[sec. 3]{Bur} or the Cohen-Macaulay "resolution of singularities" in \cite{Fa}. We decided to include this constructions here because we were not able to find an appropriate reference.

 For a Noetherian domain A define
$$
 A'= \bigcap_{\heit \wp =1}A_{\wp}
$$
to be the intersection of all localizations with respect to prime ideals of height 1.
We will say that $\dpth (A)>1$ if the condition $(S_2)$ from \cite[ch.~7, \S~17, (17.I)]{Mat} holds, i.e. it holds $\dpth (A_{\wp})\ge \inf (2, \heit (\wp ))$ for all $\wp\in \Spec (A)$.
Then one proves

\begin{lemmaB}
\label{lemma2} Assume $\dim (A)>1$. Then $A' = A$ if and only if
$\dpth (A)>1$.
\end{lemmaB}
\begin{proof} If $A' = A$, then for any non-zero non-invertible $f\in A$ we have (since $A$ is a domain)
$$
fA=\bigcap_{\heit \wp =1}fA_{\wp}= \bigcap_{\heit \wp =1}(fA_{\wp}\bigcap A),
$$
and the ideals $fA_{\wp}$ either coincide with $A_{\wp}$ or are $\wp$-primary in the rings $A_{\wp}$, since
the rings $A_{\wp}$ are Noetherian local rings of dimension one.  So, the ideals $fA_{\wp}\bigcap A$ in $A$
either coincide with $A$ or are $\wp$-primary in $A$. Thus, there is a primary decomposition for $fA$ without
 embedded components (cf. \cite[th.~4.10]{AM}) and therefore there are non-invertible non-zero divisors in $A_{\wp}/fA_{\wp}$ for any $\wp \supset f$ of height one,
 because $\dim A>1$ (cf. \cite[prop.~4.7]{AM}). Hence $\dpth (A)>1$.

Now assume $\dpth (A)>1$. If $x\in  A'$ then the set of all elements $s\in A$
such that $sx\in A$ is an ideal, not contained in any prime ideal of height 1.
Since $\dpth (A)>1$, there is a regular sequence $(a,b)$ in this ideal.
Since $a(xb) - b(xa) = 0$, it follows that $xa\in aA$, thus $x\in A$.
\end{proof}

Assume $A$ has the following property:
$$
\leqno{(*)}
\begin{array}{l}\mbox{\quad Every prime ideal of height
$1$ in the normalization of $A$ intersects $A$  in a prime ideal }
\\ \quad \mbox{of height $1$.}
\end{array}
$$
This property is satisfied for example for domains of finite type
over a field or over the integers by~\cite[prop.~5.6, corol.~5.8,
th.~5.10, th.~5.11]{AM} and by~\cite[vol.~I, ch.~V, th.~9]{ZS}.

Then for every Noetherian domain $B$ between $A$ and its
normalization with $\dpth (B)>1$  we have that $A'$ is contained in
$B$ (since $B_{\wp}\supset A_{\wp}\cap A$ for any prime ideal $\wp$
of $B$ of height one by (*) and $B=B'$ by lemma~\ref{lemma2}).

\begin{lemmaB}
\label{lemma2.5}
Assume $\dim A>1$, $A$ satisfies $(*)$ and has the property that its normalization is a finite $A$-module. Then we have

(i) $A'$ is a finite $A$-module;

(ii) $\dpth (A')>1$ and $A'$ is contained in any subdomain of the
normalization which contains $A$ and has depth greater than one;

(iii) for a  non-zero $f\in A$ we have  $A[1/f]' = A'[1/f]$.
\end{lemmaB}
\begin{proof} As we have seen above, $A'$ is contained in any subdomain of the
normalization which contains $A$ and has depth greater than 1. Since the normalization of $A$ is a finite $A$-module,
$A'$ is also a finite $A$-module.
To prove that $\dpth (A')>1$ we can argue as in the proof
of lemma~\ref{lemma2}. For any non-zero non-invertible $f\in A'$ we have
$$
fA'=\bigcap_{\heit \wp =1}fA_{\wp}= \bigcap_{\heit \wp =1}(fA_{\wp}\bigcap A'),
$$
and the ideals $fA_{\wp}$ either coincide with $A_{\wp}$ or are $\wp$-primary in the rings $A_{\wp}$, since the
rings $A_{\wp}$ are Noetherian local rings of dimension one.  So, the ideals $fA_{\wp}\bigcap A'$ in $A'$ either
coincide with $A'$ or are $\wp'$-primary in $A'$, where $\wp'$ is the prime ideal $\wp A_{\wp}\cap A'$. Note that
 $\heit (\wp' )=1$. Indeed, $\wp' \cap A=\wp$ and $\heit (\wp )=1$. If $\heit \wp'>1$ then there is a prime ideal
  $\wp_1\subset \wp'$ of height one such that $\wp_1\cap A=\wp$ (since $A'$ is integral over $A$, $\wp_1\cap A\neq 0$).
  But then $\wp_1=\wp'$ by~\cite[corol.~5.9]{AM}. Thus, there is a primary decomposition for $fA'$ without embedded
  components and therefore there are non-invertible non-zero divisors in $A'_{\wp}/fA'_{\wp}$ for any $\wp \supset f$ of height one,
  because $\dim A>1$ (cf.~\cite[prop.~4.7]{AM}). Hence $\dpth (A')>1$.

To prove (iii), first note that $A'$ is the intersection of

1) all localisations with respect to all prime ideals of height 1 that don't contain $f$;

2) a finite number of localisations with respect to prime ideals  of height 1 that  contain $f$.

Finite intersections and localisations commute, and the localisation of rings in item 2) with respect to $f$ is the quotient field $\Quot A$. So, we can omit all these components of the intersection. The rings in item 1) coincide with the localisations of $A[1/f]$ with respect to the same ideals. So, the equality $A'[1/f]=A[1/f]'$ follows.
\end{proof}

\begin{ntB} \em For $2$-dimensional domains the property $\dpth (A)>1$ is
equivalent with the property that $A$ is a Cohen-Macaulay ring, see
\cite[ch.~7, \S~17, (17.I)]{Mat}.
\end{ntB}

 Now we have the following geometric interpretation of the facts above.
Let $X$ be a two-dimensional integral scheme which is of finite type
over a field or over the integers. Let $P(X)$ be the subspace of all
points of height 1 with the restricted structure sheaf. Then the
direct image of the structure sheaf of $P(X)$ under the embedding of
$P(X)$ into $X$ is a coherent sheaf of algebras on $X$ (since
$P(X)(U)=A'$ for any affine $U=\Spec A$). Let $CM(X)$ be the
relative spectrum over $X$ of this algebra. Then, using lemma
\ref{lemma2.5} one obtains:

(i) $CM(X)$ is an integral schema, finite and birational over $X$, and $CM(X)$ is a
Cohen-Macaulay scheme.

(ii) Any integral $X$-scheme $Y$ which is finite and birational over $X$ and
which is Cohen-Macaulay, admits a unique factorization over $CM(X)$.

We will also {\em call} the scheme $CM(X)$ as the {\em
Cohen-Macaulaysation} of the scheme $X$.

\begin{ntB} \em
In the same way one can define the Cohen-Macaulaysation of a sheaf.

In the following let $X$ be an integral surface and $\cf$ a torsion free coherent sheaf. Let $i: P(X) \rightarrow X$ be as above.
From $\cf$ we define a sheaf $CM(\cf):= i_* (\cf |_{P(X)})$ on $X$.
Then one proves with similar arguments as above.

(i) $\cf$ is a Cohen-Macaulay sheaf if and only if $\cf =CM(\cf)$.

(ii) If $\cf$ is Cohen-Macaulay, and $\pi :X' \rightarrow X$ is the Cohen-Macaulaysation of $X$, then $\cf$ has a unique $\pi_*\co_{X'}$-module structure, so it comes from a Cohen-Macaulay sheaf $\cf'$ on $X'$ as $\cf =\pi_*\cf'$.

(iii) If $X$ is Cohen-Macaulay, then $\cf^{\vee}={{\cal{H}}om}_{\co_X}(\cf ,\co_X)$ (see, e.g., \cite[Lemma~3.1]{Bur}) is Cohen-Macaulay for any torsion free coherent sheaf $\cf$ on $X$. Therefore, we have that
$$\cf \subset CM(\cf) \subset \cf^{\vee \vee}$$
 (because for any open $U \subset X$ and the natural embedding $i: P(U) \hookrightarrow U$ we have
 \begin{multline*}{\Hom}_{\co_U}(i_*(\cf |_{P(U)}) ,\co_U) =  {\Hom}_{\co_U}(i_*(\cf |_{P(U)} ,i_*(\co_U |_{P(U)}))= \\ ={\Hom}_{\co_{P(U)}}(i^* i_*(\cf |_{P(U)}) ,\co_{P(U)}) =
 {\Hom}_{\co_{P(U)}}( \cf |_{P(U)}, \co_{P(U)}) =
 {\Hom}_{\co_U}(\cf |_U , \co_U) \mbox{,}
 \end{multline*} and the sheaves $\cf$, $i_*(\cf |_{P(X)})$ are torsion free) and by (i) the sheaf $CM(\cf)$ is minimal among Cohen-Macaulay sheaves $\cg \supset \cf$ such that $\cg /\cf$ is a torsion sheaf.

(iv) If $\cg \subset \cf$ is a coherent subsheaf of a torsion free coherent Cohen-Macaulay  sheaf $\cf$ then $\cg$ is Cohen-Macaulay if and only if no primary component of $\cg$ is associated to a closed point.
\end{ntB}

\noindent H. Kurke,  Humboldt University of Berlin, department of
mathematics, faculty of mathematics and natural sciences II, Unter
den Linden 6, D-10099, Berlin, Germany \\ \noindent\ e-mail:
$kurke@mathematik.hu-berlin.de$

\vspace{0.5cm}

\noindent D. Osipov,  Steklov Mathematical Institute, algebra and number theory
department, Gubkina str. 8, Moscow, 119991, Russia \\ \noindent
e-mail:
 ${d}_{-} osipov@mi.ras.ru$

\vspace{0.5cm}

\noindent A. Zheglov,  Lomonosov Moscow State  University, faculty
of mechanics and mathematics, department of differential geometry
and applications, Leninskie gory, GSP, Moscow, \nopagebreak 119899,
Russia
\\ \noindent e-mail
 $azheglov@mech.math.msu.su$

\end{document}